\documentclass[a4paper, 11pt]{amsart}
\usepackage[spacing,kerning]{microtype}   
\usepackage{mathptmx, amssymb,amscd,latexsym, eulervm,calligra,mathrsfs,amsmath,amsthm,amsfonts,tikz}
\usepackage{xurl}

\usepackage{setspace}
\usepackage{tabularx}
\usepackage{amsfonts}
\usepackage{paralist}
\usepackage{aliascnt}
\usepackage{blkarray}
\usepackage{booktabs}
\usepackage{enumitem}
\usepackage{setspace}
\usepackage[inner=2.5cm,outer=2.5cm, bottom=3.2cm]{geometry}
\usepackage{tikz, tikz-cd}
\usepackage{caption}
\usepackage{subcaption}
\usepackage{calligra,mathrsfs}
\usetikzlibrary{matrix}
\usetikzlibrary{arrows,calc}
\usepackage{url}
\usepackage{amssymb,latexsym,amsmath,amsfonts,amsthm,amscd,graphicx,url,color,stmaryrd}
\usepackage{nicematrix}

\theoremstyle{plain}
\newtheorem{mainthm}{Theorem}
\newtheorem{propIntro}[mainthm]{Proposition}
\newtheorem{corIntro}[mainthm]{Corollary}
\newtheorem{conjIntro}[mainthm]{Conjecture}

\newtheorem{thm}{Theorem}[section]
\newtheorem{prop}[thm]{Proposition}
\newtheorem{cor}[thm]{Corollary}
\newtheorem{lemma}[thm]{Lemma}
\newtheorem{que}[thm]{Question}

\theoremstyle{definition}
\newtheorem{definition}[thm]{Definition}
\newtheorem{remark}[thm]{Remark}

\newcommand{\NN}{\mathbb{N}}
\newcommand{\ZZ}{\mathbb{Z}}
\newcommand{\RR}{\mathbb{R}}

\newcommand{\kk}{\Bbbk}

\newcommand{\Gr}{\mathrm{Gr}}
\newcommand{\OGr}{\mathrm{OGr}}
\newcommand{\Mat}{\mathrm{Mat}}
\newcommand{\Spec}{\mathrm{Spec}}
\newcommand{\GL}{\mathrm{GL}}
\newcommand{\Ort}{\mathrm{O}}
\newcommand{\SO}{\mathrm{SO}}
\newcommand{\T}{\mathrm{T}}

\newcommand{\LSS}{\mathrm{LSS}}
\newcommand{\X}{\mathsf{X}}
\newcommand{\V}{\mathsf{V}}
\newcommand{\I}{\mathsf{I}}
\newcommand{\R}{\mathsf{R}}

\newcommand{\col}{\mathrm{col}}
\newcommand{\rk}{\mathrm{rk}}
\newcommand{\Span}{\mathrm{Span}}
\newcommand{\codim}{\mathrm{codim}}

\newcommand{\rka}{\mathrm{rk}_\mathrm{ani}}
\newcommand{\rki}{\mathrm{rk}_\mathrm{iso}}
\newcommand{\coa}{\mathrm{col}_\mathrm{ani}}
\newcommand{\coi}{\mathrm{col}_\mathrm{iso}}

\newcommand{\iso}{\mathrm{iso}}
\newcommand{\ani}{\mathrm{ani}}

\newcommand{\mcU}{\mathcal{U}}

\newcommand{\mcX}{\mathcal{X}}
\newcommand{\mcY}{\mathcal{Y}}
\newcommand{\mcZ}{\mathcal{Z}}

\newcommand{\mcG}{\mathcal{G}}

\newcommand{\mcL}{\mathcal{L}}
\newcommand{\mcS}{\mathcal{S}}

\newcommand{\ba}{\mathbf{a}}
\newcommand{\bb}{\mathbf{b}}
\newcommand{\bx}{\mathbf{x}}

\newcommand{\bz}{\mathbf{0}}
\newcommand{\bv}{\mathbf{v}}
\newcommand{\bw}{\mathbf{w}}

\newcommand{\ee}{\varepsilon_\mathrm{even}}
\newcommand{\eo}{\varepsilon_\mathrm{odd}}
\newcommand{\dev}{\delta_\mathrm{even}}
\newcommand{\dod}{\delta_\mathrm{odd}}
\newcommand{\gev}{\gamma_\mathrm{even}}
\newcommand{\god}{\gamma_\mathrm{odd}}

\renewcommand{\matrix}[1] {\begin{pmatrix}#1\end{pmatrix}}

\usepackage{array}
\newcommand{\PreserveBackslash}[1]{\let\temp=\\#1\let\\=\temp}
\newcolumntype{C}[1]{>{\PreserveBackslash\centering}p{#1}}
\newcolumntype{R}[1]{>{\PreserveBackslash\raggedleft}p{#1}}
\newcolumntype{L}[1]{>{\PreserveBackslash\raggedright}p{#1}}

\RequirePackage[url=false, isbn=false,giveninits=true,date=year,style=ieee-alphabetic,clearlang=true,maxbibnames=99,sorting=nyt,natbib=true,backend=biber, doi=false, url = false, eprint = false
]{biblatex}
\DeclareFieldFormat[article]{title}{{#1}}
\usepackage[colorlinks=true,linkcolor=black,urlcolor=black,citecolor=blue]{hyperref}
\usepackage{cleveref}
\newbibmacro{string+doi}[1]{%
	\iffieldundef{doi}{\iffieldundef{url}{#1}{\href{\thefield{url}}{#1}}}{\href{http://dx.doi.org/\thefield{doi}}{#1}}
}

\DeclareFieldFormat{title}{\usebibmacro{string+doi}{\mkbibemph{#1}}}
\DeclareFieldFormat[article, incollection, inproceedings, book, chapter, inbook, online]{title}{\usebibmacro{string+doi}{\mkbibemph{#1}}}
\newbibmacro*{volume+number}{%
	\printfield{volume}%
	\nopunct
	\setunit*{\space}%
	\printfield{number}%
}

\DeclareFieldFormat{journaltitle}{{#1}}
\DeclareFieldFormat{pages}{#1}
\DeclareFieldFormat{volume}{{#1}}
\DeclareFieldFormat[article]{number}{\addnbspace{#1}}
\DeclareFieldFormat[article]{volume}{\addnbspace\mkbibbold{#1}\nopunct}
\DeclareFieldFormat{date}{\mkbibparens{#1}}

\DeclareBibliographyDriver{article}{%
	\usebibmacro{bibindex}%
	\usebibmacro{begentry}%
	\usebibmacro{author/translator+others}%
	\newunit
	\usebibmacro{title}%
	\newunit
	\printlist{language}%
	\newunit\newblock
	\usebibmacro{byauthor}%
	\newunit\newblock
	\usebibmacro{journal+issuetitle}%
	\usebibmacro{issue+date}%
	\newunit
	\usebibmacro{byeditor+others}%
	\newunit
	\usebibmacro{pages}%
	\newunit
	\printfield{note}%
	\newunit\newblock
	\iftoggle{bbx:isbn}
	{\printfield{issn}}
	{}%
	\newunit\newblock
	\usebibmacro{doi+eprint+url}%
	\newunit\newblock
	\usebibmacro{addendum+pubstate}%
	\setunit{\bibpagerefpunct}\newblock
	\usebibmacro{pageref}%
	\newunit\newblock
	\iftoggle{bbx:related}
	{\usebibmacro{related:init}%
		\usebibmacro{related}}
	{}%
	\usebibmacro{finentry}
}

\DeclareBibliographyDriver{misc}{%
	\usebibmacro{bibindex}%
	\usebibmacro{begentry}%
	\usebibmacro{author/translator+others}%
	\newunit
	\usebibmacro{title}%
	\newunit
	\printlist{language}%
	\newunit\newblock
	\usebibmacro{byauthor}%
	\newunit\newblock
	\printfield{eprint}
	\usebibmacro{issue+date}%
	\usebibmacro{finentry}
}
\DeclareBibliographyDriver{online}{%
	\usebibmacro{bibindex}%
	\usebibmacro{begentry}%
	\usebibmacro{author/translator+others}%
	\newunit
	\usebibmacro{title}%
	\newunit
	\printlist{language}%
	\newunit\newblock
	\usebibmacro{byauthor}%
	\newunit\newblock
	\printfield{url}
	\usebibmacro{finentry}
}

\renewbibmacro*{publisher+location+date}{%
	\printlist{publisher}%
	\setunit*{\space}%
	\usebibmacro{date}%
	\newunit}
\makeatletter
\@namedef{subjclassname@2020}{
	\textup{2020} Mathematics Subject Classification}
\makeatother

\begin{document}

\author[L.\,Casabella, A.\,Sammartano]{Laura~Casabella and Alessio~Sammartano}
\address{Laura Casabella: Goethe-Universität Frankfurt \\ Frankfurt \\ Germany}
\email{casabell@math.uni-frankfurt.de}
\address{Alessio Sammartano: Dipartimento di Matematica \\ Politecnico di Milano \\ Milan \\ Italy}
\email{alessio.sammartano@polimi.it}

\subjclass[2020]{Primary: 05C62, 13F70, 13H10, 14C05, 15A63; Secondary: 05E40, 13B22, 13C40,   14L30}

\title{The variety of orthogonal frames}

\begin{abstract}
An orthogonal $n$-frame is an ordered set of $n$ pairwise orthogonal vectors.
The set of all orthogonal $n$-frames 
in a $d$-dimensional quadratic vector space
is  an algebraic variety $\V(d,n)$.
In this paper, we investigate the singularities and components of the variety $\V(d,n)$, 
as well as the algebraic properties of the  quadratic ideal $\I(d,n)$ generated by the orthogonality  relations, which cuts out $\V(d,n)$ set-theoretically.
By a general result of Ananyan and Hochster, these objects satisfy  increasingly strong regularity conditions when $d$ becomes large compared to $n$; 
we establish optimal effective bounds to quantify this behavior.
We   give    criteria for the ideal $\I(d,n)$ to be prime or a complete intersection, thereby resolving open problems proposed by Conca and Welker. 
We  give criteria
for the variety $\V(d,n)$ to be normal or factorial. 
We provide strong evidence in support of the conjecture that $\I(d,n)$ defines $\V(d,n)$ scheme-theoretically.
As applications, we improve the existing bounds for 
when the Lov\'asz-Saks-Schrijver ideals of  simple graphs satisfy these four properties.
\end{abstract}

\maketitle

\section{Introduction}\label{SectionIntroduction}

Let $\kk$ be a field of characteristic not 2,
and  $E$ be a quadratic vector space over $\kk$ of dimension $d$,
that is, a vector space  equipped with a non-degenerate symmetric bilinear form
$\langle -, -\rangle$.
An {\bf orthogonal $n$-frame} 
is an  $n$-tuple $(\bv_1, \ldots, \bv_n)\in E^{\oplus n}$ 
such that  
$\bv_1, \ldots, \bv_n \in E$ are pairwise orthogonal.
The set $\V(d,n) \subseteq E^{\oplus n}$ of all orthogonal $n$-frames is an algebraic variety, called the {\bf variety of orthogonal frames}.

By choosing  coordinates, we may identify $E$ with the space $E = \kk^d$ of column vectors, 
and $E^{\oplus n}$ with the space 
 $E^{\oplus n}=\Mat(d,n)$  of $d \times n$ matrices.
Let $ S = \kk\big[\Mat(d,n)\big] = \kk\big[x_{i,j}\, :\,  1 \leq i \leq d, \, 1 \leq j \leq n\big]$
be the polynomial ring.
The variety of orthogonal frames is defined, set-theoretically, by the ideal 
$$\I(d,n) = \left( \langle \bx_j, \bx_k\rangle  \, : \, 1 \leq j < k \leq n\right) \subseteq S $$ 
where $\bx_j$ denote the columns of the generic $d\times n$ matrix $(x_{i,j})$.
We also consider the ring $\R(d,n) = S/\I(d,n)$ and  the affine scheme  $\X(d,n) = \Spec\, \R(d,n)$, whose reduced scheme is $\V(d,n)$.
The goal of this paper is to understand the geometry and algebra of the variety $\V(d,n)$, as well as the ideal $\I(d,n)$ and affine scheme $\X(d,n)$.

It is important to distinguish $\V(d,n)$ from the analogous variety obtained by additionally imposing the normalization condition that $\langle \bv_j, \bv_j \rangle = 1$ for all $j$.
The latter is the variety of \emph{orthonormal} $n$-frames, also known as the \emph{Stiefel manifold}.
Despite the very similar definition, the two varieties exhibit starkly different behavior.
The Stiefel manifold is well-understood:
it is smooth, in fact,  it is a  homogeneous space;
if $n<d$, it is an irreducible variety, 
whereas if $n = d$ it is simply the orthogonal group $\mathsf{O}(n)$;
 the ${n+1 \choose 2}$    polynomials encoding the 
orthonormality relations 
form a complete intersection, and they generate  all the polynomials vanishing on the variety. 
Long important in algebraic topology and Riemannian geometry, the Stiefel manifold has in recent years also emerged as a valuable tool in  applications in optimization theory and data science, see e.g. \cite{CV,JD,LLPSW,LLT}.

In contrast, the variety of orthogonal frames $\V(d,n)$
displays considerably greater complexity,
and the primary goal of this paper is to initiate its exploration.
It is a singular variety, and it can have many components.
Moreover, it is an open problem whether $\V(d,n)$ coincides with the  affine scheme $\X(d,n)$ defined by the  orthogonality relations, 
that is, whether the ideal $\I(d,n)$ is radical.

Of particular interest, when $n\leq d$, is the \emph{principal component} of $\V(d,n)$, which is the irreducible component that generically parameterizes linearly independent orthogonal frames.
This variety too is poorly understood:
its prime ideal and its singularities are known only for special values of $d,n$.
For a striking example, we do not know all the polynomials vanishing on the set of  orthogonal bases of $\RR^6$.

We now describe our results in detail.
Since $\V(d,1) = \kk^d$, we assume $n \geq 2$ from now on.
Define the following numerical functions:

\begin{align*}
D_{\mathrm{CI}}(n) & := 
\min \left(2 \left \lceil \frac{2n+1 - \sqrt{8n+1}}{2} \right\rceil
\,,\, \,
2 \left \lceil \frac{2n - \sqrt{8n-7}-1}{2} \right\rceil +1\right),
\\
\\
D_{\mathrm{prime}}(n) & :=\min \left(2 \left \lfloor \frac{2n+1 - \sqrt{8n+1}}{2} \right\rfloor +2
\,,\, \,
2 \left \lfloor \frac{2n - \sqrt{8n-7}+1}{2} \right\rfloor +1\right),
\\
\\
D_{\mathrm{UFD}}(n)  &:=
\begin{cases}
\min \left(
2 \left\lceil \frac{ 2n+1 - \sqrt{8n-23}}{2}\right\rceil
\,,\, \,
2 \left\lceil \frac{ 2n -1- \sqrt{8n-31}}{2}\right\rceil+1
\right) & \quad \text{if } n \geq 4,\\
4=D_{\mathrm{prime}}(3) +1& \quad\text{if } n =3,\\
3=D_{\mathrm{prime}}(2)+1 & \quad\text{if } n =2.
\end{cases}
\end{align*}

Our  main theorems are the following:

\begin{mainthm}\label{ThmCI}
The ideal $\I(d,n)$ is a complete intersection if and only if $d \geq D_{\mathrm{CI}}(n)$.
In this case, the scheme $\X(d,n)$ is reduced.
\end{mainthm}

\begin{mainthm}\label{ThmPrime}
The ideal $\I(d,n)$ is prime if and only if $d \geq D_{\mathrm{prime}}(n)$.
In this case, 
 $\V(d,n)$ is a normal variety.
\end{mainthm}

Theorem \ref{ThmPrime}  solves \cite[Question 9.5]{CW},
where the authors also pointed out the difficulty of formulating a precise conjecture for this formula.
 Theorem \ref{ThmCI} solves the corresponding problem for the complete intersection property,  which is also extensively investigated in \cite{CW}.

The following result strengthens Theorem \ref{ThmPrime} in nearly all cases:
\begin{mainthm}\label{ThmUFD}
The ring $\R(d,n)$ is a unique factorization domain if $d \geq D_{\mathrm{UFD}}(n)$.
\end{mainthm}
\noindent
Indeed, we have 
 $D_{\mathrm{prime}}(n) = D_{\mathrm{UFD}}(n) $ for almost all $n \in \NN$ (in the sense that they differ in a subset of $ \NN$ of density 0).

Besides being a very natural geometric object,  a major motivation for studying 
the variety of orthogonal frames  comes from the theory of \emph{Lov\'asz-Saks-Schrijver} (LSS) ideals
\cite{CW,HMMW}.
Given a simple graph $\mathcal{G}$ on $n$ vertices, its $d$-th LSS ideal is 
$$
\LSS(d,\mathcal{G}) = 
\left( \sum_{i=1}^d x_{i,j}x_{i,k} \, : \, (j,k) \text{ is an edge of } \mathcal{G}\right).
$$ 
The $d$-th LSS ideal cuts out the variety of \emph{orthogonal representations} of $\mathcal{G}$, that is,  maps $f : [n] \to \kk^d$ such that $f(j) \perp f(k)$ whenever $j,k$ are adjacent vertices in $\mathcal{G}$.
These algebraic and geometric objects have long been studied in the combinatorial literature,
 in relation to   structural  properties of graphs  such as connectivity, Shannon capacity, vertex packing 
\cite{GLS,LLS,L79,L}.
More recently, they have played a role in connections to  
sparse determinantal varieties \cite{CW}, representation theory \cite{AA,K}, and physics \cite{Gram,HMPS}.
LSS ideals are  a vast generalization of  well-studied classes of ideals: when  $d= 1$ they recover the (monomial) edge ideals, and when $d=2$ they recover the permanental edge ideals 
\cite{HMMW}
and the parity binomial edge ideals \cite{KSW},
see also
\cite{Kumar}.

In this context, it is important to understand when LSS ideals are radical, prime, complete intersections,  and, more generally, to understand the singularities of the affine schemes they define. 
When $\kk^d$ is equipped with the standard bilinear form,
 $\I(d,n)$  is the LSS ideal of the complete graph on $n$ vertices. 
However, $\I(d,n)$ is not just a special case of LSS ideal: its properties have implications for LSS ideals of \emph{all} graphs, as we explain in the next paragraphs.

By a landmark theorem of Ananyan and Hochster \cite{AH}, if we fix $\mathcal{G}$ or $n$, and $d$ becomes sufficiently large, the scheme defined by $\LSS(d,\mathcal{G})$ 
 will have increasingly nice properties, such as complete intersection, reduced, integral, normal, or factorial. 
An important problem is to find explicit sufficient conditions in terms of $d, n, $ and $\mathcal{G}$.
Conca and Welker give sufficient conditions for the prime and complete intersection properties of  LSS ideals, 
in terms of a graph-theoretic concept called positive matching decomposition, cf. \cite[Theorem 1.3]{CW}.
More recently, Tolosa Villarreal obtains sufficient conditions for  
LSS varieties  to be normal, to have rational singularities, or to be factorial,
in terms of the positive matching decomposition and of the degeneracy number of a graph, cf. \cite[Theorems A and B]{TV}.
Considering all graphs on $n$ vertices, the results of \cite{CW} and \cite{TV}
give the following bounds in terms of $d$ and $n$.

\begin{propIntro}\label{PropPreviousBounds} Let $\mcG$ be a simple graph on $n$ vertices.
\begin{enumerate}
\item  If $d \geq 2n-3$, then 
$\LSS(d,\mathcal{G})$ is a radical complete intersection;
\item if $d \geq 2n-2$, then $\LSS(d,\mathcal{G})$ is prime;
\item if $d \geq 3n-4$, then  $S/\LSS(d,\mathcal{G})$ is a normal domain;
\item if $d \geq 3n-3$, then $S/\LSS(d,\mathcal{G})$ is a UFD.
\end{enumerate}
\end{propIntro}

We improve these  results as follows.

\begin{corIntro}\label{CorImprovedBound}
Let $\mcG$ be a simple graph on $n$ vertices.
\begin{enumerate}
\item  If $d \geq D_{\mathrm{CI}}(n)$, then 
$\LSS(d,\mathcal{G})$ is a radical complete intersection;
\item if $d \geq D_{\mathrm{prime}}(n)$, then   $S/\LSS(d,\mathcal{G})$ is a normal domain;
\item if $d \geq D_{\mathrm{UFD}}(n)$, then $S/\LSS(d,\mathcal{G})$ is a UFD.
\end{enumerate}
\end{corIntro}

Theorems \ref{ThmCI} and \ref{ThmPrime} imply that  bounds (1) and (2)  in Corollary \ref{CorImprovedBound}
are the best possible  for every $n$.
They also imply that  bounds  (1) and (2)   in Proposition \ref{PropPreviousBounds} are asymptotically sharp, since both
 $D_{\mathrm{CI}}(n), D_{\mathrm{prime}}(n)\sim 2n$.
Likewise,  Theorem \ref{ThmUFD} implies that bound (3) 
in Corollary \ref{CorImprovedBound}
is the best possible bound for nearly all $n$.
Corollary \ref{CorImprovedBound} follows from Theorems \ref{ThmCI}, \ref{ThmPrime}, \ref{ThmUFD}, since the ring $S/\LSS(d,\mathcal{G})$ can be interpreted as deformations of $\R(d,n)$.

An important open problem about the variety of orthogonal frames,
as  mentioned above,  
is whether it is defined ideal-theoretically, not just set-theoretically, by the  orthogonality relations.
This is the content of the following conjecture:

\begin{conjIntro}[\protect{\cite[Conjecture 9.6]{CW}}]\label{ConjectureRadical}
The ideal $\I(d,n)$ is radical for all $d,n$.
\end{conjIntro}

Equivalently, $\V(d,n) = \X(d,n)$, i.e., the scheme $\X(d,n)$ is reduced.
When $d= 2$, this conjecture is proved  in \cite[Theorem 1.1]{HMMW}.
Some 
 evidence in support of this conjecture is provided by Theorem \ref{ThmCI},
where it is  verified  under the assumption that $\I(d,n)$ is a complete intersection. 
Considerably stronger evidence is provided by the following theorem, where  we  prove this conjecture \emph{generically}, that is, we show that $\X(d,n)$ is reduced in a dense open set.

\begin{mainthm}\label{ThmGenericallyReduced}
The scheme $\X(d,n)$ is generically reduced for all $d,n$.
\end{mainthm}

\subsection*{Organization of the paper}
In Section \ref{SectionPreliminaries}, 
we recall basic facts about quadratic vector spaces, 
and fix the notation for the main objects of this paper, namely,
the ideal, variety, and scheme of orthogonal frames.
In Section \ref{SectionStratification},
we introduce  a stratification of $\V(d,n)$ into locally closed subsets. 
For each stratum, we show that it is a smooth quasiprojective variety, determine its connected components and dimension.
Section \ref{SectionDimension} deals with  the optimization problem that arises from the study of the dimensions of the strata.
Section \ref{SectionDegeneration} deals with degeneration problem associated to this stratification, that is, the problem of determining when a stratum belongs to the closure of another stratum.
While we do not fully solve this problem, we solve a substantial portion that  suffices to deduce the decomposition of $\V(d,n)$ into irreducible components.
Section \ref{SectionSmoothness} contains the final main ingredient,
that  is, a technique to produce smooth points of $\X(d,n)$ on sufficiently many strata. 
We achieve this by translating the problem into a calculation of certain Hilbert functions of general intersection of ideals.
In Section \ref{SectionProofs}, we use all the tools developed in Sections 
\ref{SectionStratification}, \ref{SectionDimension}, \ref{SectionDegeneration}, and \ref{SectionSmoothness}
to prove all the main results presented in this Introduction.
Finally, in Section \ref{SectionFutureDirections}, 
we discuss 
some future directions and open problems suggested by
our work.

\subsection*{Acknowledgments} The authors would like to thank Aldo Conca, Alessio D'Alì, Alessandro De Stefani, Jason McCullough, Bernd Sturmfels, and Jerzy Weyman for helpful conversations. The software Macaulay2 \cite{M2} was used for some computations related to this paper.  L.~C. would like to acknowledge the Max Planck Institute for Mathematics in the Sciences, where she was affiliated during the development of this work.

\subsection*{Funding} 
L.~C. was partially supported by the DFG priority program \emph{Combinatorial Synergies} (SPP 2458) – 539866395.
A.~S. was supported by the grant PRIN 2022K48YYP \emph{Unirationality, Hilbert schemes, and singularities}, and is a member of INdAM – GNSAGA. 

\subsection*{AI disclosure} LLMs were used solely for  proofreading the final manuscript.

\section{Preliminaries}\label{SectionPreliminaries}

In this section, 
we present some preparatory material.
First, we recall  basic terminology and  results about vector spaces equipped with a non-degenerate symmetric bilinear form.
Then, we fix  definitions and notations for the main objects of the paper.

\subsection{Ground field}
Throughout the paper,  $\kk$ denotes a field of characteristic different from 2.
In several sections, 
we assume that $\kk$ is algebraically closed, 
specifically, in Sections \ref{SubsectionIdealSchemeVariety}, \ref{SectionStratification}, \ref{SectionDegeneration},
\ref{SectionSmoothness}, and part of Section \ref{SubsectionQuadraticVectorSpaces}.
Nevertheless, all the main  results are proved without this assumption in Section \ref{SectionProofs}.
The field plays no role in Section \ref{SectionDimension}.

\subsection{Quadratic vector spaces}\label{SubsectionQuadraticVectorSpaces}
We collect some  well-known facts about symmetric bilinear forms. A standard reference on this material is	\cite[Chapter XV]{Lang}, for instance.

Consider pairs $(E,b)$ where  $E$ is a finite dimensional vector space over $\kk$ and  $b$  is a symmetric bilinear form on $E$.
Vectors $\bv, \bw \in E$ are orthogonal if $b(\bv, \bw) =0$.
A basis $\{\bv_1, \ldots, \bv_d\}$ of $E$ is  orthogonal if $b(\bv_j, \bv_k) = 0$ for all $j \ne k$,
it is  orthonormal if additionally $b(\bv_j, \bv_j) = 1$ for all $j$.
Orthogonal bases always exist.
An {\bf orthogonal $n$-frame} in $E$ is an ordered set of vectors $(\bv_1, \ldots, \bv_n)$ where $b(\bv_j, \bv_k) = 0$ for $j \ne k$.
If, additionally, $b(\bv_j,\bv_j) = 1$ for all $j$, then it is called an orthonormal $n$-frame.

A vector $\bv \in E$ is  {\bf isotropic} if $b(\bv, \bv) =0$, otherwise it is {\bf anisotropic}.

An isometry between $(E, b)$ and $(E',b')$ is an isomorphism $f:E\to E'$  such that $b(\bv, \bw) = b'(f(\bv),f(\bw))$ for all $\bv, \bw \in E$.

\begin{prop}\label{PropNonDegenerateSpaces}
 The following conditions are equivalent for  $(E,b)$:
\begin{enumerate}
\item $b$ is non-degenerate, that is, $b(\bv,\bw)=0$ for all $\bw \in E$ if and only if $\bv = \bz$;
\item the Gram matrix of $b$ with respect to a basis of $E$ is non-singular;
\item every (equivalently, any) orthogonal basis of $E$ consists of anisotropic vectors.
\end{enumerate}
\end{prop}
\noindent
If $b$ is non-degenerate, we call $(E,b)$  a quadratic vector space.
Assume this for the rest of the subsection.

Let $H\subseteq E$ be a subspace.
Its orthogonal complement $H^\perp=\{\bv \in E \,:\, b(\bv,\bw) = 0 \, \forall \, \bw \in H\}$ 
 satisfies $\dim H^\perp = \dim E - \dim H$.
 We say that $H$ is non-degenerate if the restriction of $b$ to $H$ is non-degenerate. 
 We say that  $H \subseteq E$ is isotropic if every $\bv \in H$ is isotropic,
equivalently, if $b(\bv,\bw)=0$ for all $\bv, \bw \in H$.

\begin{prop}\label{PropNonDegenerateSubspaces}
If $H$ is non-degenerate, then $H \cap H^\perp = 0$ and $H^\perp$ is also non-degenerate.
\end{prop}

\begin{thm}[Witt]\label{ThmWitt}
Let $H,K \subseteq E$ be subspaces
and $f: H \to K$ an isometry.
Then, $f$ can be extended to an isometry $g : E \to E$.
\end{thm}

For the rest of the subsection, assume $\kk=\overline{\kk}$.
Then, $E$ has orthonormal bases.

\begin{prop}\label{PropMaximalDimensionIsotropic}
Suppose  that $H$ is isotropic.
Then, $H$ is maximal with respect to inclusion among isotropic subspaces  of $E$ if and only if $\dim H =  \lfloor \frac{\dim E}{2}\rfloor$.
\end{prop}

The orthogonal group $\Ort(E) $ is the group of isometries $f:E\to E$, 
and it is also  the variety parameterizing orthonormal $d$-frames of $E$.
This variety  is a complete intersection in the space of endomorphisms of $E$, with ${d+1 \choose 2}$ equations and dimension $ {d \choose 2}$.
It is  smooth   with 2 isomorphic connected components.
The special orthogonal group $\SO(E)$ is the connected component containing the identity.

\begin{cor}\label{CorWitt}
Let $\ba_1, \ldots, \ba_r, \bb_1, \ldots, \bb_s\in E$ and $ \bv_1, \ldots, \bv_r, \bw_1, \ldots, \bw_s \in E$  be such that
\begin{itemize}
\item $\ba_j, \bv_j$ are anisotropic for all $j$,
\item $\bb_j, \bw_j$ are  isotropic for all $j$,
\item $\{\ba_1, \ldots, \ba_r, \bb_1, \ldots, \bb_s\}$ and $\{\bv_1, \ldots, \bv_r, \bw_1, \ldots, \bw_s\}$ are linearly independent and orthogonal sets.
\end{itemize}
Then there exists an isometry $g: E \to E$ such that $g(\ba_j) = c_j\bv_j$, $g(\bb_j) = \bw_j$ for all $j$,
where each $c_j $ can be chosen to be one of the two square roots of $\frac{b(\ba_j,\ba_j)}{b(\bv_j,\bv_j)}$.
\end{cor}

\begin{proof}
It follows from Theorem \ref{ThmWitt}.
\end{proof}

It follows that $\Ort(E)$ acts transitively on:
 the set of $p$-dimensional non-degenerate subspaces,
the set of orthonormal $p$-frames, 
the set of $p$-dimensional isotropic subspaces, where $p$ is any positive integer.

The {\bf orthogonal Grassmannian} $\OGr(q,E)$ parameterizes 
isotropic subspaces of $E$ of dimension $q$. 
It is a smooth projective variety  of dimension  $q(2\dim E-3q-1)/2$.
It is irreducible unless  $2q = \dim E>0$, 
in this case it has 2 connected  components,  isomorphic to each other.
See \cite[Lecture 22]{Harris} for more details.

\subsection{The ideal, variety, and scheme of orthogonal frames}\label{SubsectionIdealSchemeVariety}
We  fix the notation for the subsequent sections of the paper.
We assume that the field $\kk$ is algebraically closed.

Let $d,n\in \NN$ be  integers with $d \geq 1, n \geq 2$.
Let $E$ be a quadratic vector space of dimension $d$.
By choosing a basis, we  identify $E$ with the space $E = \kk^d$ of column vectors, 
and $E^{\oplus n}$ with the space 
 $E^{\oplus n}=\Mat(d,n)$  of $d \times n$ matrices.
Since $\kk = \overline{\kk}$ and $\mathrm{char} (\kk) \ne 2$, 
we  assume without loss of generality that  $E$ is equipped with the standard non-degenerate symmetric bilinear form
$$
 \langle \bv, \bw\rangle = \bv^\intercal \bw = v_1 w_1 + \cdots + v_d w_d.
$$
We identify 
matrices $A \in \Mat(d,n)$ and $n$-frames in $\kk^d$, i.e., ordered sets of column vectors $(\ba_1, \ldots, \ba_n)$.

Consider the coordinate ring of $  \Mat(d,n)$,
$$
S = \kk[x_{i,j} \, : \, 1 \leq i \leq d, 1 \leq j \leq n].
$$
The ideal generated by the orthogonality relations is 
$$\I(d,n) = \left( \sum_{i=1}^d x_{i,j}x_{i,k} \, : \, 1 \leq j < k \leq n\right), $$ 
the quotient algebra is denoted $\R(d,n) = S/\I(d,n)$,
the corresponding affine scheme is
$$\X(d,n) = \Spec\, \R(d,n) \subseteq \Mat(d,n),$$ 
and the   variety of orthogonal frames is
$$\V(d,n) = \Spec(S/\sqrt{\I(d,n)})    = \Big\{ A \in \Mat(d,n) \, : \, A^\intercal A \text{ is diagonal}\Big\}.$$

In this paper, the term ``variety''  refers to schemes over $\kk$ that are reduced and quasiprojective, but possibly reducible.

The following groups act on $\X(d,n)$ and $\V(d,n)$:
the orthogonal group $\Ort(d)$ by left multiplication,
the torus $\T(n)=(\kk^*)^n$  by scaling  columns,
the symmetric group $\Sigma_n$  by permuting  columns.
We denote by $G$ the group  generated by $\Ort(d), \T(n)$ and $\Sigma_n$, that is,
\begin{equation}\label{EqGroupG}
G = \Ort(d) \times \big( \T(n) \rtimes \Sigma_n) \subseteq \GL(d) \times \GL(n),
\end{equation}
where the semidirect product $\T(n) \rtimes \Sigma_n$ is the group of generalized permutation (or monomial) matrices.

\section{Stratification of the  variety of orthogonal frames}\label{SectionStratification}

In this section, we investigate  the topological structure of the variety $\V(d,n)$. 
We  construct a stratification of $\V(d,n)$ into smooth equidimensional locally closed subschemes, 
and determine the dimension and  components of the strata.
We assume $\kk = \overline{\kk}$ throughout this section.

The construction of the stratification is based on the following  numerical invariants.

\begin{definition}\label{DefNumericalInvariantsLSSMatrices}
Let $A \in \V(d,n)\subseteq \Mat(d,n)$ be an orthogonal frame.
The {\bf anisotropic column space} and the {\bf isotropic column space}  of $A$ are 
\begin{align*}
\coa(A) &=  \Span \big( \text{anisotropic columns of } A\big) \subseteq \kk^d,\\
\coi(A) &=  \Span \big( \text{isotropic columns of } A\big)\subseteq \kk^d.
\end{align*}
The {\bf anisotropic rank} and the {\bf isotropic rank}  of $A$ are
$$
\rka(A) = \dim\big( \coa(A)\big), \qquad
\rki(A) = \dim\big( \coi(A)\big).
$$
\end{definition}

\begin{remark}
Since the columns of $A \in \V(d,n)$ are orthogonal, the anisotropic columns are automatically linearly independent, so $\rka(A)$ is simply the number of anisotropic columns.
By Propositions \ref{PropNonDegenerateSpaces} and \ref{PropNonDegenerateSubspaces},
the subspaces $\coa(A)$ and $\coa(A)^\perp$ are non-degenerate and  $\coa(A) \cap \coa(A)^\perp = 0$.

The subspace $\coi(A)\subseteq \kk^d$ is isotropic, since its generators are orthogonal and isotropic.
Since
 $\coa(A) \subseteq \coi(A)^\perp$,
we have $\coa(A) \cap \coi(A)=0$ and therefore
\begin{equation}
\rka(A)+\rki(A) = \rk(A) \leq n.
\end{equation}
On the other hand, $\coi(A)$ is  an isotropic subspace of $\coa(A)^\perp$,
therefore, 
 by Proposition \ref{PropMaximalDimensionIsotropic}, we also have
$2\rki(A) \leq \dim \big(\coa(A)^\perp\big) = d-\rka(A)$ and, thus,
\begin{equation}
\rka(A)+2\rki(A) \leq d.
\end{equation}
\end{remark} 
 
\begin{definition}\label{DefStratumSpq}
Let $p,q \in \NN$ be such that $p+q \leq n$ and $ p+2q \leq d$.
Define the locus
\begin{equation}
\mcS_{p,q}= \big\{ A \in \V(d,n) \,\mid \, \rka(A) = p,\, \rki(A) = q\big\}.
\end{equation}
\end{definition}

The  variety   $\V(d,n)$ is the disjoint union of the loci $\mcS_{p,q}$.
Both  $\rka(A)$ and $\rki(A)$  are invariant under the actions of the groups $\Ort(d), \T(n), \Sigma_n$  on $\V(d,n)$, thus,
the actions restrict to each $\mcS_{p,q}$.

The goal of 
 this section is to show that $\mcS_{p,q}$ is a smooth quasiprojective variety, and to   determine its dimension and irreducible components.
In   Section \ref{SectionDegeneration}, we  study the degeneration problem for this stratification, that is, determining  when  a stratum $\mcS_{p,q}$ lies in the closure of  another  stratum $\mcS_{p',q'}$.

Each  $\mcS_{p,q}$ is typically reducible, with many isomorphic irreducible components arising from the action of $\Sigma_n$ permuting the columns.
For this reason, it is convenient to work with the following sub-loci.

\begin{definition}\label{DefStratumLpq}
Let $p,q \in \NN$ be such that $p+q \leq n$ and $ p+2q \leq d$.
Define the locus
\begin{equation}
\mcL_{p,q}= \big\{ A \in \V(d,n) \,\mid \, \rka(A) = p,\, \rki(A) = q, \text{ the first } p \text{ columns of } A \text{ are anisotropic}\big\}.
\end{equation}
More generally, 
for any subset $P \subseteq [n]$, define 
$\mcL_{P,q}$  to be the set of matrices $A \in \V(d,n)$ such that  $P$ is the set of indices of the anisotropic columns of $A$, and $\rki(A) = q$.
\end{definition}

Note that $\mcS_{p,q}$  is the disjoint union of the loci $\mcL_{P,q}$ as $P$ ranges over all subsets of $[n]$ with  $|P|=p$.
The actions of $\Ort(d)$ and $\T(n)$ restrict to each $\mcL_{P,q}$, whereas 
$\Sigma_n$ acts transitively on the set of loci $\{\mcL_{P,q}\,|\, P \subseteq [n],\, |P|=p\}$.
It follows  that 
 $\mcS_{p,q}$ is the disjoint union of the loci obtained from $\mcL_{p,q}$ by applying the action of
$\Sigma_n$, thus reducing  the study of $\mcS_{p,q}$ to that of $\mcL_{p,q}$.

The following theorem is the main result of this section.

\begin{thm}\label{ThmStructureLpq}
Let $p,q \in \NN$ be such that $p+q \leq n$ and $ p+2q \leq d$.
The locus  $\mcL_{p,q}$ is a smooth quasiprojective variety of dimension
\begin{equation}\label{EqDimLpq}
\dim \mcL_{p,q} = pd +qd +qn -\frac{1}{2}p^2-2qp-\frac{3}{2}q^2+\frac{1}{2}p-\frac{1}{2}q.
\end{equation}
If $p>0$ or $d>2q$, then $\mcL_{p,q}$ is irreducible.
If $p=0$ and $d = 2q>0$, then $\mcL_{p,q}$ has two connected components, they are irreducible and isomorphic to each other.
\end{thm}

As an immediate consequence, we obtain the following result for the strata of $\V(d,n)$.

\begin{cor}\label{CorStructureSpq}
Let $p,q \in \NN$ be such that $p+q \leq n$ and $ p+2q \leq d$.
The stratum $\mcS_{p,q}$ is a smooth quasiprojective variety of dimension
\begin{equation}\label{EqDimSpq}
\dim \mcS_{p,q} = pd +qd +qn -\frac{1}{2}p^2-2qp-\frac{3}{2}q^2+\frac{1}{2}p-\frac{1}{2}q.
\end{equation}
If $p>0$ or $d>2q$,  then $\mcS_{p,q}$ has ${n \choose p}$ connected components.
If $p=0$ and $d = 2q>0$, then $\mcS_{p,q}$ has $2{n \choose p}$ connected components.
All connected components  of $\mcS_{p,q}$ are irreducible and isomorphic to each other.
\end{cor}

\begin{proof}
This follows from Theorem \ref{ThmStructureLpq} since $\mcS_{p,q}$ is the disjoint union of the loci  $\mcL_{P,q}$ for $P\subseteq [n]$ with $|P|=p$,
and every such $\mcL_{P,q}$ is obtained from $\mcL_{p,q}$ by a permutation of the columns.
\end{proof}

We  begin the preparation for the proof of Theorem \ref{ThmStructureLpq}.
Let $\mathrm{Gr}(p,\kk^d)$ denote the Grassmannian of $p$-dimensional subspaces of $\kk^d$. 
We use the same symbol to denote  a point $H\in \mathrm{Gr}(p,\kk^d)$ and the corresponding subspace $H \subseteq \kk^d$.

The following proposition follows from Section \ref{SubsectionQuadraticVectorSpaces}.

\begin{prop}\label{PropLocusUOpen}
The locus 
$
\mcU = \big\{ H  \, \mid 
\, H \text{ is non-degenerate}\big\}\subseteq \mathrm{Gr}(p,\kk^d)
$
is open and  non-empty.
The orthogonal group $\Ort(d)$ acts transitively on $\mcU$.
\end{prop}

We  recall a basic fact about morphisms with irreducible fibers.

\begin{lemma}\label{LemmaIrreducibleComponents}
Let $f:X \to Y$ be a flat morphism of finite type of  varieties over $\kk$.
If $f$ has irreducible fibers, then it induces a bijection between the sets of irreducible components of $X$ and $Y$.
\end{lemma}
\begin{proof}
See, for example, \cite[Lemmas 5.8.15, 29.23.2, 29.25.9]{Stacks}.
\end{proof}

We are now ready to prove the main theorem of this section.

\begin{proof}[Proof of Theorem \ref{ThmStructureLpq}]
The locus 
 $\mcL_{p,q}$
 is a quasiprojective variety, since it is locally closed in 
 $\V(d,n)$:
the condition that the first $p$ columns are anisotropic is open,
the condition that the last $n-p$ columns are isotropic is closed,
and the condition that  they span a subspace of dimension $q$ is locally closed.

We consider the affine space $\Mat(d,n)$ as the product $\Mat(d,p)\times \Mat(d,n-p)$.
We are going to realize  the variety $\mcL_{p,q}\subseteq \Mat(d,n)$ as a projection of a fiber product, over the open set $\mcU \subseteq \Gr(p,\kk^d)$ of Proposition \ref{PropLocusUOpen},
of two  schemes $\mcX_\ani \subseteq \Mat(d,p)$ and $\mcX_\iso \subseteq \mcU \times  \Mat(d,n-p)$.

\underline{Step 1: construction of $\mcX_\ani$.}
First, let $\mcY = \{ A \in \Mat(d,p) \,\mid \, A^\intercal A = \mathrm{Id}\}$, where $\mathrm{Id}$ denotes the  identity matrix.
This is  the Stiefel manifold,
that is, the closed subvariety of $\Mat(d,p)$ consisting of orthonormal $p$-frames.
It is a homogeneous space, since  $\Ort(d)$ acts transitively on $\mcY$, thus, it is smooth.
It is a scheme over $\mcU$, with structure morphism  given by $A \mapsto \col(A)$, and every fiber is isomorphic to $\Ort(p)$. 
Since  $ \mcY \to \mcU$  is $\Ort(d)$-equivariant,  it is flat by \cite[Proposition 1.65 (a)]{Milne}, and therefore smooth.

Next, define
 $$
 \mcX_\ani = 
 \big\{ A \in \Mat(d,p) \,\mid \, A^\intercal A \text{ is diagonal and invertible}
 \big\},
 $$
 that is, the locally closed subset of $\Mat(d,p)$ consisting of matrices whose columns are orthogonal and anisotropic.
 As before,  $\mcX_\ani$  is a scheme over $\mcU$.

There is  a morphism of $\mcU$-schemes 
 $\mcY \times \T(p) \to \mcX_\ani$ defined by  $(A,D) \mapsto AD$,
 where  $\T(p)\subseteq \GL(p)$ is the standard torus.
This  is the quotient map by the group $(\ZZ_2)^p$, acting diagonally on $\mcY \times \T(p)$ by scaling columns by $\pm 1$.
Indeed, the map is clearly surjective, 
since any anisotropic orthogonal frame $\bb_1, \ldots, \bb_p \in \kk^d$ can be scaled to an orthonormal set, 
and  each  $\bb_i$ must be scaled by one of the two square roots of $\langle\bb_i, \bb_i\rangle$.
Since the finite group $(\ZZ_2)^p$ acts freely on the  quasiprojective variety  $\mcY \times \T(p)$,
the quotient map  is étale by \cite[Section II.7, Theorem]{Mumford},
thus, 
$\mcX_\ani$ is  smooth and of the same dimension as  $\mcY \times \T(p)$.
It also follows that every  fiber of $\mcX_\ani \to \mcU$ is isomorphic to  $\big(\Ort(p)\times \T(p)\big)/(\ZZ_2)^{p}$.
This fiber is smooth, since it is a homogeneous space, and it is irreducible, since the restriction to the irreducible  component
 $\SO(p) \times \T(p) \to \big(\Ort(p)\times \T(p)\big)/(\ZZ_2)^{p}$ is surjective.
Finally, since  $\mcX_\ani \to \mcU$  is a morphism of smooth varieties with equidimensional fibers, it is flat, and therefore smooth,  by miracle flatness 
\cite[Lemma 10.128.1]{Stacks},
of relative dimension
\begin{equation}\label{EqRelativeDimXani}
\dim \mcX_\ani-\dim \mcU =  \dim \Ort(p)+ \dim \T(p) =  {p \choose 2} +p  = {p+1 \choose 2}.
\end{equation}

\underline{Step 2: construction of $\mcX_\iso$.}
Define
$$
\mcX_{\iso} = 
\big\{ (H,B) \in \mcU \times \Mat(d,n-p) \,\mid \, B^\intercal B = 0,\, \rk(B) = q,\, \col(B) \perp H  
\big\}.
$$
This is a locally closed subvariety of $\mcU \times \Mat(d,n-p)$, and a scheme over $\mcU$ via the projection map.
Observe that $B^\intercal B = 0$ amounts to the fact that the columns of $B$ are isotropic and orthogonal.

Let $\zeta$ denote the restriction to $\mcU$ of the tautological rank $p$ subbundle over $\Gr(p,\kk^d)$,
and $\zeta^\perp$ the orthogonal complement, of rank $d-p$.
Consider the  orthogonal Grassmann bundle $\OGr(q,\zeta^\perp)$, which is a  projective scheme over $\mcU$ parameterizing isotropic subbundles of $\zeta^\perp$ of rank $q$.
The $\kk$-points of $\OGr(q,\zeta^\perp)$ are pairs $(H,K)$ where $H \in \mcU$ and $K \subseteq H^\perp$ is isotropic with $\dim K = q$.
By Corollary \ref{CorWitt}, $\Ort(d)$ acts transitively on $\OGr(q,\zeta^\perp)$,
so $\OGr(q,\zeta^\perp)$ is a homogeneous space and a smooth variety.

The structure morphism $\OGr(q,\zeta^\perp) \to \mcU$ is smooth, with  fibers isomorphic to 
$\OGr(q,\kk^{d-p})$.
If  $2q < d-p$,
it follows by Lemma \ref{LemmaIrreducibleComponents} that $\OGr(q,\zeta^\perp)$ is irreducible.
If $p=0$ and $d=2q$,  then  $\mcU$ is a point, so 
$\OGr(q,\zeta^\perp) \simeq \OGr(q,\kk^{d})$ has two irreducible components, disconnected and isomorphic to each other.
Finally, if $p >0$ and $d = 2q+p$, we claim that $\OGr(q,\zeta^\perp)$ is  irreducible.
For this, consider the other projection $\OGr(q,\zeta^\perp)\to \OGr(q,\kk^d)$, defined by $(H,K) \mapsto K$, which is an $\Ort(d)$-equivariant morphism, in particular, flat by  \cite[Proposition 1.65 (a)]{Milne}.
The fiber over $K$ is $\mcU \cap \Gr(p,K^\perp)$, in particular, it is irreducible. 
The desired conclusion follows again by Lemma \ref{LemmaIrreducibleComponents}.

There is a morphism of $\mcU$-schemes 
$\mcX_\iso \to \OGr(q,\zeta^\perp)$ 
defined by  $(H,B) \mapsto (H,\col(B))$.
Again, this morphism is  $\Ort(d)$-equivariant and flat.
Every fiber is isomorphic to the open subvariety of full rank matrices in $\Mat(q,n-p)$:
specifically,
if $K = \col(C)$ for some $C \in \Mat(d,q)$ with $\rk(C) = q$, 
the fiber over $(H,K)$ is $\{(H,CF) \, \mid \, F \in \Mat(q,n-p), \, \rk(F) =q\}$.
This implies that the morphism is smooth, with irreducible fibers of dimension $q(n-p)$.

It follows  by the previous two paragraphs that  $\mcX_\iso$ is irreducible 
if $p>0$ or $d>2q$, 
whereas if  $p=0$ and $d = 2q>0$ then it has two connected components, both irreducible.
It also  follows that  the structure morphism 
$\mcX_\iso \to \mcU$ is  smooth, of relative dimension
\begin{equation}\label{EqRelativeDimXiso}
\dim \mcX_\iso-\dim\mcU =  \dim \Mat(q,n-p)+ \dim 
\OGr(q,\kk^{d-p}) =
q(n-p) + q\frac{2(d-p)-3q-1}{2}.
\end{equation}

\underline{Step 3: conclusion.}
Finally,
consider the fiber product $\mcZ = \mcX_{\ani}\times_\mcU \mcX_{\iso}$,
which is a locally closed subvariety of   $\mcU \times \Mat(d,n)$.
By the previous two steps, 
$\mcZ$  is smooth over $\mcU$, with relative dimension 
\begin{equation}
\dim \mcZ-\dim \mcU =  {p +1\choose 2} +
q(n-p) + q\frac{2(d-p)-3q-1}{2},
\end{equation}
therefore, $\mcZ$ is smooth, of dimension 
\begin{equation}
\dim \mcZ = 
p(d-p) +  {p +1\choose 2} +
q(n-p) + q\frac{2(d-p)-3q-1}{2}.
\end{equation}
Since $\mcX_{\ani} \to \mcU$ is flat with (geometrically) irreducible fibers,
it follows by base change that $\mcZ \to \mcX_{\iso}$ is also flat with irreducible fibers, 
In particular, if $p>0$, then $\mcX_{\iso}$ and thus $\mcZ$ are irreducible.
On the other hand, if $p=0$, then $\mcZ \simeq \mcX_{\iso}$.
In conclusion, $\mcZ$ is irreducible 
if $p>0$ or $d>2q$, 
whereas if  $p=0$ and $d = 2q>0$ then it has two irreducible components, necessarily disconnected, by smoothness.
The two   components are isomorphic, 
since $\mcZ\to \OGr(q, \zeta^\perp)$  is $\Ort(d)$-equivariant and $\Ort(d)$ acts transitively on $\OGr(q, \zeta^\perp)$.

Finally, the projection  $\mcU \times \Mat(d,n) \to \Mat(d,n)$ restricts to an isomorphism $\mcZ \to \mcL_{p,q}$,
where the inverse map is given by letting $H\in \mcU$ be the span of the first $p$ columns of a matrix $A \in \mcL_{p,q}$.
\end{proof}

\begin{remark}\label{RemTransitiveIrreducibleComponentsSpq}
The group $G$ of \eqref{EqGroupG}  acts transitively on the set of irreducible components of $\mcS_{p,q}$, for each $p,q$.
Indeed, the symmetric group $\Sigma_n$ acts transitively on the set of loci $\mcL_{P,q} $ such that $|P| = p$.
If  $p>0$ or $d>2q$, these loci are precisely the irreducible components of $\mcS_{p,q}$, and the claim follows.
When  $p=0$ and $d = 2q>0$, the claim follows from the additional observation 
 that  $\Ort(d)$ acts transitively on the set of irreducible components of  $\mcL_{p,q}$, 
as explained in Step 3 of the proof of Theorem \ref{ThmStructureLpq}.
\end{remark}

We conclude the section by observing that some loci $\mcL_{p,q}$ are homogeneous spaces.

\begin{cor}\label{CorLpqHomogeneous}
If $p+q=n$ or $q= 0$, 
then $\Ort(d) \times \T(n)$ acts transitively on 
the locus $\mcL_{p,q}$.
\end{cor}
\begin{proof}
If $p+q=n$, 
then  $\mcL_{p,q}$ is the set of all frames $(\bv_1,\ldots,\bv_p, \bw_1, \ldots, \bw_q) \in \Mat(d,n)$ such that 
$\bv_j$ is anisotropic for all $j$,
$\bw_j$ is isotropic for all $j$,
and $\bv_1,\ldots,\bv_p, \bw_1, \ldots, \bw_q$ are linearly independent and orthogonal.
The conclusion follows from Corollary \ref{CorWitt}.
If $q = 0$,
then  $\mcL_{p,q}$ is the set of all frames $(\bv_1,\ldots,\bv_p, \bz, \ldots, \bz) \in \Mat(d,n)$ such that  the
$\bv_j$'s are orthogonal and anisotropic, and again we can apply Corollary \ref{CorWitt}.
\end{proof}

\section{Dimension of strata}\label{SectionDimension}

In this section, we analyze the behavior of the numerical function 
\begin{equation}\label{EqFunctionStrataDimension}
\sigma(p,q) = pd +qd +qn -\frac{1}{2}p^2-2qp-\frac{3}{2}q^2+\frac{1}{2}p-\frac{1}{2}q
\end{equation}
on the domain
\begin{equation}\label{EqQuadrilateral}
\Delta(d,n) = \big\{
(p,q) \in \ZZ^2 \,\, : \,\,
p \geq 0, \,\,
 q \geq 0, \,\,
p+q \leq n, \,\,
 p+2q \leq d
\big\}
\end{equation}
where  $d,n$ are fixed  integers with $d \geq 1, n \geq 2$.
By Section \ref{SectionStratification}, 
$\sigma(p,q)$ is equal to the dimension of each irreducible component of the stratum $\mcS_{p,q}$, whereas    $\Delta(d,n)$
indexes  the non-empty strata of $\V(d,n)$.
Our goal is to determine the values $(p,q) \in \Delta(d,n)$ where $\sigma(p,q)$ attains the absolute maximum.

\begin{prop}\label{PropStrictlyIncreasing}
The function $\sigma$ is strictly increasing in $\Delta(d,n)$
with respect to both variables $p$ and $q$.
\end{prop}

\begin{proof}
It suffices to consider the first differences
\begin{align*}
\sigma(p,q)-\sigma(p-1,q) &= 
d-p-2q+1 \geq 1\\
\sigma(p,q)-\sigma(p,q-1) &= d+n-2p-3q+1 = (d-p-2q)+(n-p-q)+1\geq 1. \qedhere
\end{align*}
\end{proof}

The following subset
\begin{equation}\label{EqUpperBoundary}
\Omega = \left\{
\Big(\min(n-q,d-2q)\,,\,q\Big) 
\, : \,  0 \leq q \leq \min\left(n,\left\lfloor\frac{d}{2}\right\rfloor\right)
\right\}
\end{equation}
is  the \emph{upper boundary}  of the domain $\Delta(d,n)$.
Indeed, $\Omega$  is precisely the set of maximal elements in $\Delta(d,n)$
with respect to the componentwise order of $\NN^2$.
This set
 consists of one or two segments on the lines  $p+q=n$ and $p+2q=d$,
which intersect at the point $(p,q) = (2n-d,d-n)$.
Specifically, define
\begin{equation}\label{EqSegment1}
\Omega_1 = 
\big\{(n-q, q) \, : \,  0 \leq q \leq \min (d-n,n)
\big\},
\end{equation}
\begin{equation}\label{EqSegment2}
\Omega_2 = \left\{
(d-2q,q) \, : \,  \max (d-n+1,0) \leq q \leq \left\lfloor\frac{d}{2}\right\rfloor
\right\},
\end{equation}
then $\Omega = \Omega_1 \sqcup \Omega_2$.
Whether only one or both segments are non-empty depends on $d$ and $n$, as follows:

\begin{equation}\label{EqTable3CasesSegments}
\begin{tabular}{|c|c|}
\hline
$2n \leq d $
&
$\Omega_1 = \big\{(n-q, q) \, : \,  0 \leq q \leq n \big\} $ 
\\
&
$\Omega_2 = \varnothing$
\\
\hline
$d =2n-1$
&
$\Omega_1 = \big\{(n-q, q) \, : \,  0 \leq q \leq n-1 \big\} $ 
\\
&
$\Omega_2 = \varnothing$
\\
\hline
$n \leq d < 2n-1$ 
&
$\Omega_1 = \big\{(n-q, q) \, : \,  0 \leq q \leq d-n\big\}$
\\
&
$\Omega_2 = \left\{(d-2q,q) \, : \,  d-n+1 \leq q \leq \left\lfloor\frac{d}{2}\right\rfloor \right\}$
\\
\hline
$ d< n$
&
$\Omega_1 = \varnothing$
\\ 
& 
$\Omega_2 = \left\{(d-2q,q) \, : \,  0 \leq q \leq \left\lfloor\frac{d}{2}\right\rfloor\right\}$
\\
\hline
\end{tabular}
\end{equation}

See Figure \ref{FigDomain}  for some illustrations.

\begin{prop}\label{PropRestrictionSegment1}
The restriction of the function $\sigma$ to the segment $\Omega_1$ is strictly decreasing in $q$.
\end{prop}
\begin{proof}
Substituting $p = n-q$ in $\sigma(p,q)$ we obtain a linear polynomial in $q$, with leading coefficient $-1$
\begin{align*}
\sigma(n-q,q)& =  -\frac{1}{2}(n-q)^2-2q(n-q)-\frac{3}{2}q^2+ (n-q)d +qd +qn +\frac{1}{2}(n-q)-\frac{1}{2}q\\
& 
= -q+dn+\frac{1}{2}n-\frac{1}{2}n^2 \qedhere
\end{align*}
\end{proof}

\begin{prop}\label{PropRestrictionSegment2}
The restriction of the function $\sigma$ to the segment $\Omega_2$ is non-decreasing in $q$,
and it is strictly increasing in the interval
$\max (d-n+2,0) \leq q \leq \left\lfloor\frac{d}{2}\right\rfloor$.
\end{prop}
\begin{proof}
Substituting $p = d-2q$ in $\sigma(p,q)$ we obtain 
\begin{align*}
\sigma(d-2q,q)& = -\frac{1}{2}(d-2q)^2-2q(d-2q)-\frac{3}{2}q^2+ (d-2q)d +qd +qn +\frac{1}{2}(d-2q)-\frac{1}{2}q\\
&= \frac{1}{2}q^2 + \left(n-d-\frac{3}{2}\right) q +\frac{d^2+d}{2}.
\end{align*}
This is a quadratic polynomial in $q$, with positive leading coefficient and minimum attained at  $
q = d-n + \frac{3}{2}.
$
Thus, it attains the same value at $q= d-n+1$ and $q= d-n+2$, and it is strictly increasing afterwards.
Since the segment $\Omega_2$ begins at $q= d-n+1$ or later, the claim follows.
\end{proof}

Consider the integer points 
\begin{equation}\label{EqEndpoints}
P_1 = \Big(\min(d,\,n)\,,\,0\Big) \qquad
\text{and}
\qquad
P_2 =
 \left(
 d -2\left\lfloor\frac{d}{2} \right\rfloor
 \,,\, 
 \left\lfloor\frac{d}{2} \right\rfloor
 \right).
\end{equation}
Observe that $P_1$ is the initial endpoint of the upper boundary $\Omega$, while $P_2$ is the final endpoint of $\Omega$  only if $d \leq 2n$.
Observe also that $P_1 \ne P_2$ unless $d = 1$.

\begin{prop}\label{PropMaximum}
The  maximum of  $\sigma(p,q)$ in $\Delta(d,n)$ is attained  at one or both of the points $P_1, P_2$, and nowhere else.
More precisely, we have the following cases:

\begin{equation*}
\begin{tabular}{|C{4cm}|C{4cm}|C{4cm}|}
\hline
\multicolumn{3}{|c|}{d \emph{ even}}
\\
\hline
$d > 2n+1 - \sqrt{8n+1}$
&
\emph{maximum } 
$= nd - {n \choose 2}$
&
\emph{ attained only at } $P_1$
\\
\hline
$d = 2n+1 - \sqrt{8n+1}$
&
\emph{maximum } 
$= nd - {n \choose 2}$
&
\emph{attained at both } $P_1, \, P_2$
\\
\hline
$d < 2n+1 - \sqrt{8n+1}$
&
\emph{maximum} 
$> nd - {n \choose 2}$
&
\emph{ attained only at } $P_2$
\\
\hline
\end{tabular}
\end{equation*}

\vspace*{4mm}

\begin{equation*}
\begin{tabular}{|C{4cm}|C{4cm}|C{4cm}|}
\hline
\multicolumn{3}{|c|}{d \emph{ odd}}
\\
\hline
$d > 2n - \sqrt{8n-7}$
&
\emph{maximum } 
$= nd - {n \choose 2}$
&
\emph{attained only at } $P_1$
\\
\hline
$d = 2n - \sqrt{8n-7}$
&
\emph{maximum } 
$= nd - {n \choose 2}$
&
\emph{ attained at both } $P_1, \, P_2$
\\
\hline
$d < 2n - \sqrt{8n-7}$
&
\emph{maximum } 
$> nd - {n \choose 2}$
&
\emph{attained only at } $P_2$
\\
\hline
\end{tabular}
\end{equation*}
\vspace*{1mm}
\end{prop}
\begin{proof}
By Proposition \ref{PropStrictlyIncreasing}, the  maximum of $\sigma(p,q)$ in $\Delta(d,n)$ is attained only in the upper boundary $\Omega= \Omega_1\sqcup \Omega_2$.
By Proposition \ref{PropRestrictionSegment1}, if $\Omega_1 \ne \varnothing$, 
then $P_1 = (n,0)$,
the  maximum of $\sigma(p,q)$ in $\Omega_1$  is attained precisely at $P_1$, and it is equal to $\sigma(n,0) = nd - {n \choose 2}$.
By Proposition \ref{PropRestrictionSegment2}, if $\Omega_2 \ne \varnothing$, 
the  maximum of $\sigma(p,q)$ in $\Omega_2$  is attained precisely at $P_2 $ in all cases, except when  $\Omega_2$ consists of the two elements with  $q= d-n+1,d-n+2$;
in this case, $\sigma(p,q)$ attains the same value at both elements.

We are going to  verify the statements of the proposition in each case  of \eqref{EqTable3CasesSegments}.

First, suppose  $2n-1 \leq d$.
Then, $\Omega_2 = \varnothing$ and
 the global maximum of $\sigma(p,q)$ is attained only at $P_1$.
This implies the conclusions of the proposition: if $d$ is even, then $d \geq 2n-1 > 2n+1 -\sqrt{8n+1}$,
 whereas if 
$d$ is odd, then
 $d \geq 2n-1>2n-\sqrt{8n-7}$,
since $n \geq 2$ by assumption.

Next, suppose   $d<n$, then $\Omega_1 = \varnothing$.
The segment $\Omega_2$ consists of one element if and only if $d = 1$.
It consists of the two elements with  $q= d-n+1,d-n+2$ if and only if $d-n+1 = 0$ and $d-n+2 = \lfloor \frac{d}{2}\rfloor$,
that is, if and only if $ (d,n) = (2,3), (3,4)$.
If $ (d,n) \ne (2,3), (3,4)$,
 the global maximum of $\sigma(p,q) $ is attained only at $P_2 $.
If $ (d,n) = (2,3), (3,4)$,
then $\Omega_2=\{P_1,P_2\}$, with $P_1\ne P_2$,
and $\sigma(p,q)$ 
attains the global maximum at both points.

The inequality
$$
2\left(\sigma(P_1) - \left(nd - {n \choose 2}\right)\right)
= d^2+d-2dn+n^2-n 
=
\left(n-d-\frac{1}{2}\right)^2- \frac{1}{4}\geq 0
$$
implies that $ \sigma(P_1) \geq dn-{n \choose 2}$, 
and, moreover, that the inequality is strict when $n>d+1$.
Since $\sigma(P_2)\geq \sigma(P_1)$ by Proposition \ref{PropRestrictionSegment2}, 
we conclude that 
$\sigma(P_2) > dn-{n \choose 2}$
 if $n > d+1$ or $\sigma(P_2) > \sigma(P_1)$.
 By the previous paragraph, we have $\sigma(P_2) = \sigma(P_1)$ if and only if 
$d=1$ or 
 $(d,n) = (2,3), (3,4)$.
In conclusion,
 we have 
$\sigma(P_2) > nd - {n \choose 2}$ 
when $(d,n) \ne (1,2), (2,3), (3,4)$,
whereas
$\sigma(P_2) = nd - {n \choose 2}$ holds
when $(d,n) = (1,2), (2,3), (3,4)$.

Assume $d$ is even.
If $n \geq 6$, then $d<n \leq 2n+1-\sqrt{8n+1}$ holds, and the conclusion follows.
For each $n = 2, \ldots, 5$, a direct calculation shows that 
$d <  2n+1-\sqrt{8n+1}$ holds for all even integers $d$ such that $d <n$,
except when $(d,n) = (2,3)$, in this case we have  $d =  2n+1-\sqrt{8n+1}$.

We proceed similarly for odd $d$.
If $n \geq 7$, then $d<n \leq 2n-\sqrt{8n-7}$ holds, and the conclusion follows.
For  $n = 2, \ldots, 6$, we have
$d <  2n-\sqrt{8n-7}$  for all odd integers $d$ such that $d <n$,
except when $(d,n) = (1,2), (3,4)$,
 in these cases we have  $d =  2n-\sqrt{8n-7}$.

Thus, in both cases $d$ even or odd, 
the conclusions of the proposition are always verified.

Finally, suppose $n \leq d < 2n-1$,
thus,
 $\Omega_1, \Omega_2 \ne \varnothing$.
The last element of $\Omega_1$ and the first element of $\Omega_2$, ordered with respect to $q$-coordinates, satisfy
$$
\sigma\big(n-(d-n),d-n\big)> \sigma\big(d-2(d-n+1),d-n+1\big)
$$
by the same argument as in  Proposition \ref{PropRestrictionSegment2}, since both points 
 belong to the line $p+2q=d$.
 It follows that the maximum of $\sigma(p,q)$ can only be attained at one or 
both of $P_1,P_2$.

We have $\sigma(P_1) =  nd - {n \choose 2}$.
If $d$ is even, then 
$
\sigma(P_2) = \sigma\left(0\,,\, \frac{d}{2}\right)   = \frac{d(d+4n-2)}{8}.
$
We have
\begin{align*}
 \frac{d(d+4n-2)}{8} < dn-{n \choose 2} & \Leftrightarrow   d^2+(-4n-2)d+4n^2-4n < 0
\\
& \Leftrightarrow  2n+1 - \sqrt{8n+1} < d < 2n+1 + \sqrt{8n+1}.
\end{align*}
Since  $d<2n-1$ by assumption, this is equivalent to 
$d >  2n+1 - \sqrt{8n+1}$.
The same equivalences hold replacing $<$ with $=$.
Thus,  the conclusions of the proposition are  verified.

Similarly, if $d$ is odd, then  $\sigma(P_2) = \sigma\left(1\,,\,\frac{d-1}{2}\right)  = \frac{d^2+4dn-4n+7}{8}$, and we have
\begin{align*}
 \frac{d^2+4dn-4n+7}{8} < dn-{n \choose 2}
  & 
  \Leftrightarrow d^2-4dn +4n^2-8n+7 < 0
  \\
    & 
  \Leftrightarrow 
2n - \sqrt{8n-7} < d < 2n + \sqrt{8n-7}.
\end{align*}
Since $2n> d$, this is equivalent to 
$
d >2n - \sqrt{8n-7},
$
and the conclusion follows.
\end{proof}

\begin{figure}[h]
	\centering
	\begin{subfigure}[b]{0.32\textwidth}
		\centering
		\begin{tikzpicture}[scale=0.65]
			\coordinate (1) at (0,0);
			\coordinate (2) at (4,0);
			\coordinate (3) at (0,4);
			\draw[fill=gray!15!white,thick,draw=black] (1)--(2)--(3)--cycle;
			
			\draw[->] (-0.2,0) -- (4.7,0) node[right] {\footnotesize$p$};
			\draw[->] (0,-0.2) -- (0,4.5) node[above] {\footnotesize$q$};
			
			\foreach \x in {0,0.5,1,...,4} 
			\pgfmathsetmacro\z{int(4-\x)}
			\foreach \y in {0,0.5,...,\z} 
			\filldraw[gray!60!white] (\x,\y) circle (0.7pt);
			
			\filldraw[white] (4,0.5) circle (1pt);
			
			\foreach \i in {0,0.5,...,4}
			\pgfmathsetmacro\z{int(100-2*(\i))}
			\filldraw[black] (4-\i,\i) circle (1.5pt) node[above right] {\tiny $\z$};
		\end{tikzpicture}
\centering				\caption{ $d=16, n=8$ \\ $\Omega_1 \neq \varnothing, \, \Omega_2 = \varnothing$}
	\end{subfigure}
	\begin{subfigure}[b]{0.32\textwidth}
		\centering
		\begin{tikzpicture}[scale=0.52]
			\coordinate (1) at (0,0);
			\coordinate (2) at (8,0);
			\coordinate (3) at (0,5);
			\coordinate (4) at (6,2);
			\draw[fill=gray!15!white,thick,draw=black] (1)--(2)--(4)--(3)--cycle;
			
			\draw[->] (-0.2,0) -- (8.7,0) node[right] {\scriptsize $p$};
			\draw[->] (0,-0.2) -- (0,5.5) node[above] {\scriptsize $q$};
			
			\foreach \y in {0,0.5,...,4.5} 
			\pgfmathsetmacro\z{int(6-\y)}
			\foreach \x in {0,0.5,...,\z} 
			\filldraw[gray!60!white] (\x,\y) circle (0.7pt);
			
			\foreach \y in {0.5,1,...,4.5} 
			\pgfmathsetmacro\z{int(9-2*\y)}
			\foreach \x in {0,0.5,1,...,\z} 
			\filldraw[gray!60!white] (\x,\y) circle (0.7pt);
			
			\filldraw[white] (8,0.5) circle (1pt);

			\filldraw[gray!60!white] (6.5,0) circle (0.7pt);
			\filldraw[gray!60!white] (7,0) circle (0.7pt);
			\filldraw[gray!60!white] (7.5,0) circle (0.7pt);
			
			\filldraw[gray!60!white] (5.5,2) circle (0.7pt);
			\filldraw[gray!60!white] (4.5,2.5) circle (0.7pt);
			\filldraw[gray!60!white] (3.5,3) circle (0.7pt);
			\filldraw[gray!60!white] (2.5,3.5) circle (0.7pt);
			
			\filldraw[black] (8,0) circle (1.5pt) node[above right] {\tiny $200$};
			\filldraw[black] (7.5,0.5) circle (1.5pt) node[above right] {\tiny $199$};			
			\filldraw[black] (7,1) circle (1.5pt) node[above right] {\tiny $198$};
			\filldraw[black] (6.5,1.5) circle (1.5pt) node[above right] {\tiny $197$};
			\filldraw[black] (6,2) circle (1.5pt) node[above right] {\tiny $196$};
			\filldraw[black] (5,2.5) circle (1.5pt) node[above right] {\tiny $195$};
			\filldraw[black] (4,3) circle (1.5pt) node[above right] {\tiny $195$};
			\filldraw[black] (3,3.5) circle (1.5pt) node[above right] {\tiny $196$};
			\filldraw[black] (2,4) circle (1.5pt) node[above right] {\tiny $198$};
			\filldraw[black] (1,4.5) circle (1.5pt) node[above right] {\tiny $201$};
			\filldraw[black] (0,5) circle (1.5pt) node[above right] {\tiny $205$};
			
		\end{tikzpicture}
\centering		\caption{ $d=20, n=16$ \\ $\Omega_1, \Omega_2 \neq \varnothing$}
	\end{subfigure}
	\begin{subfigure}[b]{0.32\textwidth}
		\centering
		\begin{tikzpicture}[scale=0.65]
			\coordinate (1) at (0,0);
			\coordinate (2) at (6,0);
			\coordinate (3) at (0,3);
			\draw[fill=gray!15!white,thick,draw=black] (1)--(2)--(3)--cycle;
			
			\draw[->] (-0.2,0) -- (6.7,0) node[right] {\footnotesize$p$};
			\draw[->] (0,-0.2) -- (0,3.5) node[above] {\footnotesize$q$};
			
			\foreach \y in {0,0.5,...,3} 
			\pgfmathsetmacro\z{int(6-2*\y)}
			\foreach \x in {0,0.5,...,\z} 
			\filldraw[gray!60!white] (\x,\y) circle (0.7pt);
			
			\filldraw[white] (0.5,3) circle (1pt);
			
			\foreach \i in {0,0.5,...,3}
			\pgfmathsetmacro\z{int(78+(2)*(\i^2)+(5)*\i}
			\filldraw[black] (6-2*\i,\i) circle (1.5pt) node[above right] {\tiny $\z$};
		\end{tikzpicture}
\centering		\caption{ $d=12, n=16$ \\ $\Omega_1=\varnothing, \, \Omega_2 \neq \varnothing$}
	\end{subfigure}
	\caption{The domain $\Delta(d,n)$ for various values of $(d,n)$. Each integer point $(p,q)$ on the upper boundary $\Omega$ is labeled with the dimension $\sigma(p,q)$.} \label{FigDomain}
\end{figure}
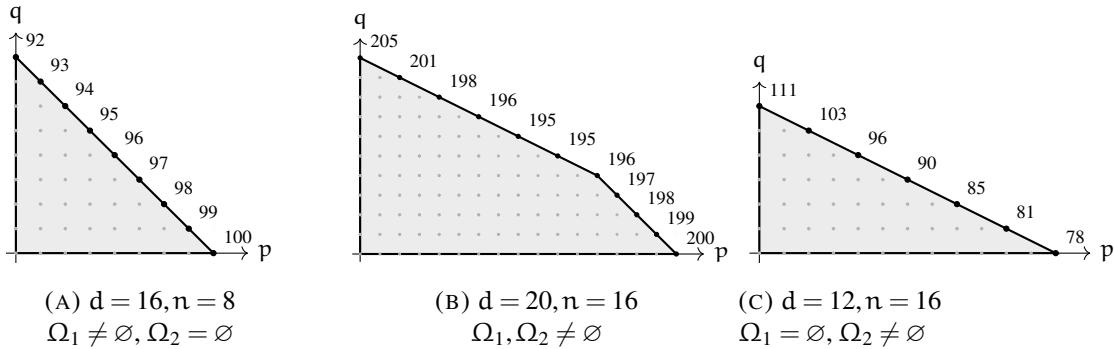

\section{Degeneration problem and irreducible components}\label{SectionDegeneration}

In this section, 
we   consider the degeneration problem  associated with the  stratification of $\V(d,n)$ introduced in Definition \ref{DefStratumSpq}.
We assume $\kk = \overline{\kk}$ in this section.

\begin{definition}
We endow the set 
\begin{equation*}
\Delta(d,n) = \big\{
(p,q) \in \ZZ^2 \,\, : \,\,
p \geq 0, \,\,
 q \geq 0, \,\,
p+q \leq n, \,\,
 p+2q \leq d
\big\}
\end{equation*}
 with the partial order defined by 
\begin{equation}\label{EqPartialOrder}
(p,q) \preceq (p',q')\,\, \Leftrightarrow \,\, \mcS_{p,q} \text{ lies  in the closure}\footnote{
Unless specified otherwise,
the closure operation always refers to the closure of a subset in $\Mat(d,n)$.
}
\text{ of  } \mcS_{p',q'}.
\end{equation}
We write $(p,q) \prec (p',q')$ if $(p,q) \preceq (p',q')$ and $(p,q) \ne (p',q')$.
\end{definition}

Understanding the poset structure of $\Delta(d,n)$ amounts to understanding the behavior of the pair 
   $(\rka(A),\rki(A))$ under  deformations of matrices $A\in \V(d,n)$.
   
\begin{remark}\label{RemIrreducibleComponentsMaximal}
By Remark \ref{RemTransitiveIrreducibleComponentsSpq}, 
if one irreducible component of $\mcS_{p,q}$ lies in the closure of $ \mcS_{p',q'}$, 
then the whole stratum $\mcS_{p,q}$ lies in the closure of $ \mcS_{p',q'}$.
Thus, an equivalent definition of the partial order is
\begin{equation*}
(p,q) \preceq (p',q')\,\, \Leftrightarrow \,\, \mcL_{p,q} \text{ lies  in the closure of  } \mcS_{p',q'}.
\end{equation*}
It also follows that 
the irreducible components of    $\V(d,n)$ are exactly the closures of the irreducible components of the strata $\mcS_{p,q}$ with $(p,q) \in \Delta(d,n)$ maximal
with respect to $\preceq$.
\end{remark}

We begin with some  necessary conditions.

\begin{prop}\label{PropNecessaryConditionsPoset}
Assume that $(p,q) \prec (p',q')$. Then,
\begin{enumerate}
\item $\sigma(p,q) < \sigma(p',q')$,
\item $p \leq p'$,
\item $p+q\leq p'+q'$.
\end{enumerate}
\end{prop}

\begin{proof}
By assumption, we have the inclusion $\mcS_{p,q}\subseteq \overline{\mcS_{p',q'}} \subseteq \Mat(d,n)$.
To prove (1), it suffices to observe that no irreducible component of  $\mcS_{p,q}$ is dense in any irreducible component of $\mcS_{p',q'}$,
therefore  $\dim \mcS_{p,q}<\dim{\mcS_{p',q'}}$.
Indeed, if this were false,  by Remark \ref{RemTransitiveIrreducibleComponentsSpq}
we would have $\overline{\mcS_{p,q}}= \overline{\mcS_{p',q'}}$
and therefore ${\mcS_{p,q}}\cap \mcS_{p',q'}\ne \varnothing$, contradiction.

Statement 
(2) follows from the lower-semicontinuity of $\rka(A)$:
if the columns of $A$ indexed by a  subset $P\subseteq[n]$ are anisotropic, 
then the same is true for any $A'$ in an open subset of $\Mat(d,n)$ containing $A$, hence, 
for some $A' \in\mcS_{p',q'}$.
Similarly,
 (3) follows  from the lower-semicontinuity of $\rk(A)$.
\end{proof}

\begin{remark}\label{RemPrincipalComponent}
By Proposition \ref{PropNecessaryConditionsPoset} (2),  $P_1 = \big(\min(d, n),\,0\big)\in \Delta(d,n)$ is maximal with respect to $\preceq$.
\end{remark}

This observation motivates the following definition.

\begin{definition}\label{DefPrincipalComponent}
Suppose $d \geq n$.
The closure $\overline{\mcS_{n,0}}$ is called 
the  {\bf principal component} of $\V(d,n)$.
\end{definition}

This locus is indeed irreducible by Corollary \ref{CorStructureSpq},
and it is a component of $\V(d,n)$ by Remark \ref{RemIrreducibleComponentsMaximal}.
This is the most natural irreducible component of  $\V(d,n)$, 
since it parameterizes \emph{linearly independent} orthogonal $n$-frames.
Moreover, its codimension is equal to the number of equations of $\V(d,n)$, that is, 
\begin{equation}\label{EqCodimPrincipal}
\dim\big( \overline{\mcS_{n,0}} \big)= \sigma(n,0) = nd - {n \choose 2} = \dim \big(\Mat(d,n)\big) - \mu \big(\I(d,n)\big),
\end{equation}
where $\mu$ denotes the number of minimal generators of an ideal.

We now describe some sufficient conditions for comparability in $\Delta(d,n)$.
We show that the partial order $\preceq$ refines the componentwise order in $\NN^2$.

\begin{prop}\label{PropDeformP}
If $(p,q), (p+1,q) \in \Delta(d,n)$, then 
 $(p,q) \preceq (p+1,q)$.
\end{prop}

\begin{proof}
Let $A \in \mcS_{p,q}$. 
We are going to realize  $A$ as a limit of matrices in $\mcS_{p+1,q}$.

Since $\rk(A) = p+q<p+1+q\leq n$ by assumption,
the columns of $A$ are not linearly independent.
Specifically, there must be a linear relation among the isotropic columns of $A$, 
since all anisotropic columns are linearly independent and $\coa(A) \cap \coi(A) = 0$,
cf. Section \ref{SectionStratification}.

The subspace $\coa(A)^\perp \subseteq \kk^d$ is  non-degenerate by Proposition \ref{PropNonDegenerateSubspaces},
thus, by Proposition \ref{PropMaximalDimensionIsotropic},
a maximal isotropic subspace of $\coa(A)^\perp $ has dimension equal to
$$
\left\lfloor\frac{\dim \big(\coa(A)^\perp\big)}{2}\right\rfloor= \left\lfloor\frac{d-p}{2}\right\rfloor.
$$
The subspace $\col(A)^\perp \subseteq \coa(A)^\perp$ has dimension equal to $d-p-q$.
Since $p+2q<p+1+2q \leq d$ by assumption, 
we obtain
$2(d-p-q)> d-p$ and therefore $d-p-q > \left\lfloor\frac{d-p}{2}\right\rfloor$.
It follows that $\col(A)^\perp$ is not an isotropic subspace, that is, 
there exists an anisotropic vector $\bw \in \col(A)^\perp$.

Let $\bv$ be an isotropic column of $A$ that is a linear combination of the remaining isotropic columns, and let $A_\varepsilon$ be the matrix obtained replacing $\bv $ with $\bv + \varepsilon \bw$, with $\varepsilon \in \kk$.
Then, we have $A_\varepsilon \in \mcS_{p+1,q}$ for all $\varepsilon\ne 0$, and $\lim_{\varepsilon\to 0} A_\varepsilon = A$.
\end{proof}

\begin{prop}\label{PropDeformQ}
If $(p,q), (p,q+1)\in \Delta(d,n)$, then
 $(p,q) \preceq (p,q+1)$.
\end{prop}

\begin{proof}
We proceed similarly to Proposition \ref{PropDeformP}, showing that every  $A \in \mcS_{p,q}$ is a limit of matrices in $\mcS_{p,q+1}$.
By assumption, $\rka(A) = p, \rki(A) = q$, and $p+2q = p+2(q+1)-2 \leq d-2$.
The subspace $\coa(A)^\perp \subseteq \kk^d$  is non-degenerate,
and a maximal isotropic subspace of $\coa(A)^\perp$ has dimension 
$$
\left\lfloor \frac{d-\rka(A)}{2} \right\rfloor \geq 
\left\lfloor \frac{2 \rki(A)+2}{2} \right\rfloor 
>
\rki(A),
$$
thus, the isotropic subspace $\coi(A) \subseteq \coa(A)^\perp$ is not maximal.
It follows that there exists   $\bw \in \coa(A)^\perp \setminus \coi(A)$ such that $\coi(A) + \Span(\bw)$ is an isotropic subspace.
As before, there is an isotropic column  $\bv$ of $A$ that is a linear combination of the remaining isotropic columns.
Letting $A_\varepsilon$ be the matrix obtained replacing $\bv $ with $\bv + \varepsilon \bw$, 
 we have $A_\varepsilon \in \mcS_{p,q+1}$ for all $\varepsilon\ne 0$, and $\lim_{\varepsilon\to 0} A_\varepsilon = A$.
\end{proof}

The main result of this section is the following.

\begin{thm}\label{ThmIrreducibleDecomposition}
The irreducible components of $\V(d,n)$ are the closures of the irreducible components of the strata $\mcS_{p,q}$ such that $(p,q) \in \Omega$ and $\sigma(p,q)\geq nd - {n \choose 2}$.
\end{thm}

\begin{proof}
By Remark \ref{RemIrreducibleComponentsMaximal}, 
the irreducible components of $\V(d,n)$ are the closures of the irreducible components of the strata $\mcS_{p,q}$ such that $(p,q)$ is maximal in $\Delta(d,n)$ with respect to $\preceq$.
By Propositions \ref{PropDeformP} and \ref{PropDeformQ},
such $(p,q)$ is also maximal in $\Delta(d,n)$ with respect to the componentwise order of $\NN^2$, thus,
we must have $(p,q) \in \Omega$.
By Krull's height theorem, every minimal prime of $\I(d,n)$ has height at most $\mu(\I(d,n))= {n \choose 2}$, that is,
every irreducible component of $\V(d,n)$ has dimension at least $dn-{n\choose 2}$.

For the converse, suppose first that 
 $d \geq n$.
Then, $\Omega_1 \ne \varnothing$ and $P_1 = (n,0)$, cf. \eqref{EqTable3CasesSegments} and \eqref{EqEndpoints}.
The point $P_1 \in \Delta(d,n)$ is maximal, as observed above, 
and it satisfies $P_1 \in \Omega_1$ and  $\sigma(P_1) = dn-{n\choose 2}$.
Any other $(p,q) \in \Omega_1$ has $\sigma(p,q) < dn-{n\choose 2}$ by Proposition \ref{PropRestrictionSegment1}.

It remains to consider the elements
$(p,q) \in \Omega_2$ such that $\sigma(p,q)\geq nd - {n \choose 2}$.
Ordering them  as in \eqref{EqSegment2},
they
have non-decreasing values of $\sigma(p,q)$ by Proposition \ref{PropRestrictionSegment2}, and strictly decreasing values of  $p$.
It follows by Proposition \ref{PropNecessaryConditionsPoset} that they are all incomparable.
Moreover, they are not bounded by any of the remaining  $(p',q')\in\Omega$,
by Proposition \ref{PropNecessaryConditionsPoset},
 since they satisfy $\sigma(p',q')\leq dn-{n \choose 2} \leq \sigma(p,q)$.
 Thus, the maximal elements of $\Delta(d,n)$ are precisely 
the elements $(p,q) \in \Omega$ such that $\sigma(p,q)\geq nd - {n \choose 2}$.

If $d<n$, then $\Omega = \Omega_2$, and the  argument of the previous paragraph yields the desired conclusion.
\end{proof}

\begin{remark}\label{RemLocationMaximals}
Using Proposition \ref{PropRestrictionSegment1} and \ref{PropRestrictionSegment2}
and the fact that 
$\sigma(n,0) = nd - {n \choose 2}$, 
we see that condition $\sigma(p,q)\geq nd - {n \choose 2}$ holds  in at most two regions of the boundary $\Omega$:
\begin{itemize}
\item if $\Omega_1 \ne \varnothing$, then it holds for
the initial endpoint $(n,0) $ and no other point of $\Omega_1$;
\item in $\Omega_2$, it  holds in a (possibly empty) final segment of  the parametrization \eqref{EqSegment2} of
$\Omega_2$.
\end{itemize}
Compare, for instance, with the three examples of Figure \ref{FigDomain};
the values of $\sigma(n,0)$ are  $\sigma(8,0) = 100$ in (A), $\sigma(16,0)=200$ in (B), $\sigma(16,0)=72$ in (C).
\end{remark}

As a byproduct, we obtain the following dichotomy.

\begin{cor}\label{CorDichotomyOmega}
If $(p,q) \in \Omega$, then either $(p,q) \preceq (n,0) \in \Delta(d,n)$ or $(p,q)$ is maximal with respect to $\preceq$.
\end{cor}

\begin{proof}
Suppose $(p,q)$ is not maximal in $\Delta(d,n)$ with respect to $\preceq$.
Let $(p',q') \in \Delta(d,n)$ be  maximal 
 with respect to $\preceq$ such that $(p,q) \prec (p',q')$.
 In particular, we have $(p',q') \in \Omega$.
 By Proposition \ref{PropNecessaryConditionsPoset}, it follows that $p \leq p'$.
 However, since $(p,q)\in \Omega$, it is maximal with respect to the componentwise order, therefore, we must have $q > q'$ and $p < p'$.
By Remark \ref{RemLocationMaximals}, it follows that $(p',q')=(n,0)$.
\end{proof}

\section{Smoothness}\label{SectionSmoothness}

The goal of this  section is to prove that the scheme $\X(d,n)$ of orthogonal frames possesses an abundance of smooth points.
We assume $\kk = \overline{\kk}$ in this section.

Recall the definition of the upper boundary $\Omega$ \eqref{EqUpperBoundary} of the domain $\Delta(d,n)$ \eqref{EqQuadrilateral}.
The set $\Omega$ consists of the pairs $(p,q) \in \Delta(d,n)$ that are maximal with respect to the componentwise order in $\NN^2$.
By Propositions \ref{PropDeformP} and \ref{PropDeformQ}, 
$\Omega$ contains all the pairs $(p,q) \in \Delta(d,n)$ that are maximal
with respect to the order $\preceq$ \eqref{EqPartialOrder},
but, in general, it may contain many more pairs.

The main result of this section is the following:

\begin{thm}\label{ThmSmoothPointEveryStratum}
For every $(p,q) \in \Omega$, the locus $\mcL_{p,q}$ contains a point that is smooth on   $\X(d,n)$.
\end{thm}

In order to prove Theorem \ref{ThmSmoothPointEveryStratum},
we  establish a somewhat unexpected connection with the topic of generic Hilbert functions.
More precisely, 
we show that a crucial calculation required in the proof is equivalent 
to computing the Hilbert function  
  of certain  general intersection of ideals, in the sense of  \cite{CavigliaMurai}.
Typically,
in
problems about generic Hilbert functions,
 it is easy to determine a natural guess but very hard to prove that it is correct (and occasionally it is wrong).
Our case is still tractable, and we are able to deduce the desired statement
 from a classical genericity property of the Veronese embedding.
 
Theorem \ref{ThmSmoothPointEveryStratum}
 is equivalent to the condition  that the 
dimension of the tangent space to  $\X(d, n)$ at some point of $\mcL_{p,q}$ is  (less than or) equal to the local dimension of $\X(d, n)$ at that point.
We begin by estimating the latter.

\begin{prop}\label{PropDimensionAtStratumLpq}
Let $(p,q) \in \Omega$.
The dimension of $\X(d,n)$ at any point of $\mcL_{p,q}$
is at least
$$
\max \left( dn- {n \choose 2},\, \sigma(p,q)\right).
$$
\end{prop} 

\begin{proof}
Clearly, 
 the local dimension  of $\X(d,n)$ at any point of $\mcL_{p,q}$ is at least equal to  $\dim \mcL_{p,q} =\sigma(p,q)$.
 On the other hand, as already observed in the proof of Theorem \ref{ThmIrreducibleDecomposition},
 Krull's height theorem implies that the 
local dimension of   $\X(d,n)$ at any point
is at least $dn - {n \choose 2}$.
\end{proof}

We denote by $\Theta(d,n)$ 
the Jacobian matrix of the affine scheme $\X(d,n)$.
This matrix has ${n\choose 2}$ rows, indexed by pairs $(h,k)$ with $1\leq h<k\leq n$, corresponding to the  generators $\langle
\bx_h,\bx_k\rangle= \sum_{\iota=1}^dx_{\iota,h}x_{\iota,k}$ of  $\I(d,n)$;
it has $dn$ columns, indexed by pairs $(i,j)$ with $1 \leq i \leq d, 1 \leq j \leq n$, corresponding to the variables $x_{i,j}$ of $S$.
The entry in row $(h,k)$ and column $(i,j)$ is 
\begin{equation}\label{EqJacobianMatrixEntry}
\Theta(d,n)_{(h,k),(i,j)}= \frac{\partial }{\partial x_{i,j}}
\left(
 \sum_{\iota=1}^dx_{\iota,h}x_{\iota,k} 
 \right)
 = 
\begin{cases}
x_{i,k} & \text{if }\, j =h\,,\\
x_{i,h} & \text{if }\, j =k\,,\\
0 & \text{if }\, j \ne h,k\,.\\
\end{cases}
\end{equation}
By the Jacobian criterion, 
a point $A \in \X(d,n)$ is smooth if and only if 
the rank of $\Theta(d,n)$, 
 evaluated at $A$, 
 is greater than or equal to the local codimension of $\X(d,n) \subseteq \Mat(d,n)$ at $A$.
 If $A \in \mcL_{p,q}$ with $(p,q) \in \Omega$,
using  Proposition \ref{PropDimensionAtStratumLpq},
it suffices to show
\begin{equation}\label{EqRankJacobianAtA}
\rk \big(\Theta(d,n)(A)\big) \geq \min\left({n \choose 2},\, dn- \sigma(p,q)\right).
\end{equation}

The following observation allows to reduce the proof of \eqref{EqRankJacobianAtA} to the case $p=0$.

\begin{prop}\label{PropReductionToP0}
Suppose $A \in \mcL_{p,q}\subseteq \X(d,n)$ satisfies \eqref{EqRankJacobianAtA}.
Then, the matrix 
$$
A' = 
\matrix{1& 0 \\ 0 & A}\in \mcL_{p+1,q}\subseteq \X(d+1,n+1)
$$
also satisfies \eqref{EqRankJacobianAtA}.
\end{prop}
\begin{proof}
Up to reordering rows and columns, the new Jacobian matrix, evaluated at $A'$,  has the following block decomposition
$$
\Theta(d+1,n+1)(A')
=
\begin{pNiceMatrix}[first-row] 
\shortstack{  $(1,j)$ \\   $ 2 \leq j \leq n+1$ \\ }
&
\shortstack{  $(i,1)$ \\   $1 \leq i \leq d+1$ \\ }
&
\shortstack{  $(i,j)$ \\ $2 \leq i \leq d+1, 2 \leq j \leq n+1$ \\ } 
\\
\mathrm{Id}_{n}   & \ast &       0    \\
\ast   & 0    &       \Theta(d,n)(A)   
\end{pNiceMatrix}
\quad
\begin{array}{l}
 (1,k) \\ (h,k),\, h\geq 2
\end{array}
$$
where we marked the indices of the rows and columns relative to each of the blocks.
The top-left block is an $n\times n$ identity matrix, 
since the only non-zero entries of $\Theta(d+1,n+1)$ in this region
are equal to $x_{1,1}$, for $j=k$, and this entry is  1 in $A'$.
The top-right block vanishes, since the only non-zero entries of $\Theta(d+1,n+1)$ in this region
are of the form $x_{i,1}$ with $i \geq 2$, and they vanish on $A'$.
The bottom-right block is equal to $\Theta(d,n)(A)$, since it involves the old rows and columns.
All these assertions follow easily from formula 
\eqref{EqJacobianMatrixEntry}
and the structure of $A'$.
We conclude that 
$$\rk \big(\Theta(d+1,n+1)(A')\big)  = \rk\big(\Theta(d,n)(A)\big)+n,$$
thus, the left-hand side of \eqref{EqRankJacobianAtA} increases by $n$ when replacing 
$(d,n,p,q)$ with $(d+1,n+1,p+1,q)$ and   $A$ with $A'$.
The right-hand side also increases by $n$:
this is obvious for ${n \choose 2}$, and a straightforward calculation for $dn-\sigma(p,q)$.
\end{proof}

In the next two results, 
we prove \eqref{EqRankJacobianAtA} when $p=0$.
We denote by  $\nu \in \kk$ a fixed square root of $-1$.

\begin{prop}\label{PropSmoothPointP0QN}
If $(0,q) \in \Omega_1$,
 then $\mcL_{0,q}$ contains a smooth point of $\X(d,n)$.
\end{prop}
\begin{proof}
Since $(0,q) \in \Omega_1$,
we have
$ n=q$,
and $2n = 2q\leq d$, cf. \eqref{EqSegment1}.
The following matrix $A\in \mcL_{0,n}$
$$
A = 	\matrix{\mathrm{Id}_{n }\\\nu \mathrm{Id}_{n }\\0}	\in \X(d,n)
	$$
verifies the claim. 
Indeed, evaluate  $\Theta(d,n)$ at $A$, and consider the submatrix formed by columns $(i,j)$ with $1 \leq i < j \leq n$.
By \eqref{EqJacobianMatrixEntry},
this submatrix is the identity matrix, of size ${n\choose 2}$,
so  \eqref{EqRankJacobianAtA} is proved.
\end{proof}

\begin{prop}\label{PropSmoothPointP0Q2D}
Let $(0,q) \in \Omega_2$, then $\mcL_{0,q}$ contains a smooth point of $\X(d,n)$.
\end{prop}
\begin{proof}
Since $(0,q) \in \Omega_2$,
we have
$d=2q$ and $q \leq n$; cf. \eqref{EqSegment2}.
Equation \eqref{EqRankJacobianAtA} becomes 
\begin{equation}\label{EqRankJacobianAtAQ2D}
\rk \big(\Theta(2q,n)(A)\big) \geq
  \min\left({n \choose 2},\, nq-{q \choose 2}\right).
\end{equation}
We are going to prove this inequality for a matrix $A \in  \mcL_{0,q}$ of the form
$$
A = \matrix{B \\ \nu B}
$$
with $B =(b_{i,j})\in \Mat(q,n)$ a general matrix.
Consider the Jacobian matrix $\Theta(2q,n)(A)$ evaluated at $A$, and let 
 $\Theta'$ denote the ``half'' submatrix 
where we consider all the rows $(h,k)$, with $1 \leq h < k \leq n$,
and only the columns $(i,j)$ with $1 \leq i \leq q$ and $1 \leq j \leq n$
(corresponding to the top half of $A$).
By \eqref{EqJacobianMatrixEntry}, the entries of this matrix are
\begin{equation}\label{EqEntryThetaPrime}
\Theta'_{(h,k),(i,j)} = 
\begin{cases}
b_{i,k}& \text{ if }\,j = h,\\
b_{i,h} & \text{ if }\,j = k,\\
0 & \text{ if }\,j \ne h,k. 
\end{cases}
\end{equation}
We will show that $\Theta'$ satisfies the lower bound \eqref{EqRankJacobianAtAQ2D}.

Now we give a different interpretation for this matrix. 
Let $P = \kk[Z_1, \ldots, Z_n]$ be a standard graded polynomial ring and $R =  P/(Z_1^2, \ldots, Z_n^2)$.
Let $\ell_i = \sum_{\gamma =1}^n b_{i,\gamma}Z_\gamma \in P$ for $i = 1,\ldots, q$; 
these are the $q$ general linear forms encoded by the rows of $B$.
For simplicity, we use $\ell_1, \ldots, \ell_q$ to denote the linear forms in both $P$ and $R$.
Let $L = (\ell_1, \ldots, \ell_q) \subseteq R$ be the ideal generated by these forms.

We use the symbol $[\cdot]_\cdot$ to denote graded components.
The vector space $[L]_2\subseteq [R]_2$ is spanned  by the quadrics 
$\ell_i Z_j$ with $1 \leq i \leq q, 1 \leq j \leq n$.
Write these quadrics in terms of the monomial basis of $[R]_2$: that is, form the matrix whose columns correspond to $ Z_hZ_k$ with $ 1 \leq h < k \leq n$,
rows correspond to $\ell_i Z_j$ with $1 \leq i \leq q, 1 \leq j \leq n$,
and the entry 
in row $(i,j)$, column $(h,k)$ is
the coefficient of $Z_hZ_k$ in $\ell_i Z_j$.
This matrix is precisely the transpose of $\Theta'$,
thus, 
 $\dim_\kk [L]_2 = \rk (\Theta')$. 

Thus, it suffices to  show the inequality
$\dim_\kk [L]_2 \geq \min \left( {n \choose 2}, qn - {q \choose 2}\right) $.
Considering the preimage   $J =  (\ell_1, \ldots, \ell_q)+(Z_1^2, \ldots, Z_n^2)\subseteq P$ of $L$, since $\dim_\kk[J]_2 = \dim_\kk[L]_2 + n$, the inequality is equivalent to
\begin{equation}\label{EqRankJ}
\dim_\kk [J]_2 \geq \min \left( {n \choose 2}+n, qn - {q \choose 2}+n\right) = 
\min \left( {n+1 \choose 2}, (q+1)n - {q \choose 2}\right).
\end{equation}
Since  $q \leq n$,
we may regard $\ell_1, \ldots, \ell_q$ as the images of $Z_1, \ldots, Z_q$ under a general  change of coordinates. 
Applying the inverse change of coordinates, we obtain the ideal  $K= (Z_1, \ldots, Z_q)+(h_1^2, \ldots, h_n^2) \subseteq P$, where $h_1, \ldots, h_n \in P$ are general linear forms. 
Then, \eqref{EqRankJ} is equivalent to
\begin{equation}\label{EqRankK}
\dim_\kk [K]_2 \geq 
\min \left( {n+1 \choose 2}, (q+1)n - {q \choose 2}\right)
\end{equation}
We now translate the problem again, by going  modulo the variables $Z_1, \ldots, Z_q$.
Let $r = n-q$, denote by $Y_1, \ldots, Y_r$ the images of $Z_{q+1}, \ldots, Z_n$ in the polynomial ring  $Q = \kk[Y_1, \ldots, Y_r] = P/(Z_1, \ldots, Z_q)$, 
 use the same symbols
 $h_1, \ldots, h_n \in Q $
to denote the images of the general linear forms,
and let $H = (h_1^2, \ldots, h_n^2)\subseteq Q$ be the image of $K$.
We have  
$$ \dim_\kk [K]_2 -\dim_\kk[H]_2 =\dim_\kk[(Z_1, \ldots, Z_q)]_2
=qn - {q \choose 2}
$$
since the only linear syzygies among  $Z_1, \ldots, Z_q$ are the Koszul syzygies.
Subtracting this quantity 
from the two terms in the  right hand side of \eqref{EqRankK} we obtain
\begin{align*}
{n+1 \choose 2} - \dim_\kk\big[(Z_1, \ldots, Z_q)\big]_2 &= \dim_\kk \big[\kk[Y_{1}, \ldots Y_r]\big]_2 =  {r+1\choose 2},\\
(q+1)n-{q \choose 2} -  \dim_\kk\big[(Z_1, \ldots, Z_q)\big]_2 &= n,
\end{align*}
thus, \eqref{EqRankK} is equivalent to 
\begin{equation}\label{EqJRankH}
\dim_\kk [H]_2 = 
\min \left( {r+1 \choose 2}\,,\,\, n\right).
\end{equation}
This amounts to the statement that  if $h_1, \ldots, h_n \in \kk[Y_1, \ldots, Y_r]$ are general linear forms,
then $ h_1^2, \ldots, h_n^2$ are either linearly independent or they generate 
$\kk[Y_1, \ldots, Y_r]_2$.
But this follows from the well-known fact that the
Veronese embedding preserves the general position  property.
More precisely, the map $\big[\kk[Y_1, \ldots, Y_r]\big]_1 \to \big[\kk[Y_1, \ldots, Y_r]\big]_2$,
 defined by $h \mapsto h^2$, induces a morphism $\Phi: \mathbb{P}^{r-1} \to \mathbb{P}^{{r+1 \choose 2}-1}$
that is given in coordinates by 
$$
\Phi\big([\alpha_1 : \cdots :\alpha_r]\big)= [\alpha_1^2 : 2\alpha_1\alpha_2 :  \alpha_2^2 : 2 \alpha_1\alpha_3:  \cdots :\alpha_r^2].
$$
Since $\mathrm{char}(\kk) \ne 2$, $\Phi$ is the composition of second Veronese   embedding  $\nu_2$, given by
$$
\nu_2\big([\alpha_1 : \cdots :\alpha_r]\big)= [\alpha_1^2 : \alpha_1\alpha_2 :  \alpha_2^2 :  \alpha_1\alpha_3:  \cdots :\alpha_r^2],
$$
 with an invertible projective linear transformation.
Since $h_1, \ldots, h_n$ are general,
we may assume that 
the points $[h_1], \ldots, [h_n] \in \mathbb{P}^{r-1}$ are in general quadratic position, 
that is, any subset of $n' \leq {r+1\choose 2}$ points imposes independent conditions on quadrics. 
This means precisely that 
any subset of $n' \leq {r+1\choose 2}$ points among 
 $\nu_2([h_1]), \ldots, \nu_2([h_n]) \in
\mathbb{P}^{{r+1 \choose 2}-1}$
imposes independent conditions on linear forms,
and therefore the same holds for $\Phi([h_1]), \ldots, \Phi([h_n])$.
In conclusion, $ h_1^2, \ldots, h_n^2$ are either linearly independent or they generate 
$\kk[Y_1, \ldots, Y_r]_2$.
\end{proof}

\begin{proof}[Proof of Theorem \ref{ThmSmoothPointEveryStratum}]
A point $A \in \X(d,n)$ is smooth if  $\eqref{EqRankJacobianAtA}$ holds.
If $(p,q) \in \Omega(d,n)$ with $p>0$,  then, by \eqref{EqUpperBoundary},
 $(p-1,q) \in \Omega(d-1,n-1)$.
 Thus, by Proposition \ref{PropReductionToP0}, we may assume $p=0$, and the statement is proved in Propositions \ref{PropSmoothPointP0QN} and \ref{PropSmoothPointP0Q2D}.
\end{proof}

 \section{Proof of the main results}\label{SectionProofs}
 
In this section, we build on the results developed in the previous sections and derive the proofs of the main theorems of the paper. 
Here, the only assumption on the field is that $\mathrm{char}(\kk) \ne 2$.
 
We begin with an arithmetic lemma.

\begin{lemma}\label{LemmaArithmetic}
Let $\ee, \eo\in  \RR$ be two real numbers such that $|\ee-\eo|\leq 1$, and let $d \in \ZZ_{\geq 2}$.
\begin{enumerate}
\item The inequality
$$
d \geq 
\min \left(2 \left \lceil \frac{\ee}{2} \right\rceil
\,,\, \,
2 \left \lceil \frac{\eo-1}{2} \right\rceil +1\right)
$$
is equivalent to  $d \geq \ee$ if $d$ is even, to  $d \geq \eo$ if $d$ is odd.
\item
The inequality 
$$
d \geq \min \left(2 \left \lfloor \frac{\ee}{2} \right\rfloor +2\,,\, \,
2 \left \lfloor \frac{\eo+1}{2} \right\rfloor +1\right)
$$
is equivalent to  $d > \ee$ if $d$ is even, to  $d > \eo$ if $d$ is odd.
\end{enumerate}
\end{lemma}
\begin{proof}
We begin with the proof of (1). 
The number $\dev=
2 \left \lceil \frac{\ee}{2} \right\rceil$
is the least even integer that is greater than or equal to $\ee$,
and
$\dod=
2 \left \lceil \frac{\eo-1}{2} \right\rceil +1$
is the least odd integer that is greater than or equal to $\eo$.
Moreover, it is easy to see that the assumption 
$|\ee-\eo|\leq 1$
guarantees that 
$|\dev-\dod|\leq 1$.
This implies the desired conclusion.

The proof of (2) is similar:
$\gev = 2 \left \lfloor \frac{\ee}{2} \right\rfloor +2$
is the least even integer that is strictly greater than $\ee$,
$\god =  2 \left \lfloor \frac{\eo+1}{2} \right\rfloor +1$
is the least odd integer that is strictly greater than $\eo$,
and the assumption 
$|\ee-\eo|\leq 1$
guarantees that 
$|\gev-\god|\leq 1$.
\end{proof}

Next, we recall Serre's conditions and some basic facts about passing to the algebraic closure.
Let $A$ be a ring and $k\in \NN$. 
We say that $A$ has Serre's property $R_k$ if the local ring $A_\mathfrak{p}$ is regular for every prime ideal $\mathfrak{p}$ of $A$ with height at most $k$.
We say that $A$ has Serre's property $S_k$ if $A_\mathfrak{p}$  has depth at least $\min(k, \dim A_\mathfrak{p})$ for every prime ideal $\mathfrak{p}$ of $A$.

\begin{lemma}\label{LemmaAlgebraicClosure}
Let $A$ be a Noetherian standard graded $\kk$-algebra, and let $B = A \otimes_\kk \overline{\kk}$.
Then,
\begin{enumerate}
\item if $B$ is reduced, a domain, or a normal domain,  then so is $A$;
\item if $B$ has property $R_k$, respectively, $S_k$, then so does $A$;
\item $A$ is  equidimensional, Cohen-Macaulay, Gorenstein, or a  complete intersection if and only if $B$ has the same property;
\end{enumerate}
\end{lemma}

\begin{proof}
The natural map $A \to A \otimes_\kk \overline{\kk}$ is faithfully flat. 
Statements (1) and (2) follow  by \cite[Section 10.164]{Stacks}.
In (3), 
the statement about equidimensionality follows from 
\cite[II, Ex. 3.20]{Hartshorne};
the remaining statements follow 
 by faithful flatness since they can be checked, 
 for example, in terms of the minimal free resolutions of $A$ and $B$.
\end{proof}

Finally, we recall that Serre's conditions behave well under deformations.

\begin{lemma}\label{LemmaDeformation}
Let $A$ be a Noetherian standard graded $\kk$-algebra, $f\in A$ a regular homogeneous element of positive degree.
Then,
\begin{enumerate}
\item if $A/(f)$ is a domain, then $A$ is a domain;
\item if $A/(f)$ is Cohen-Macaulay  and has property $R_k$, then  $A$ has property $R_k$.
\end{enumerate}
\end{lemma}
\begin{proof}
For (1), assume that $(f) \subseteq A$ is a prime ideal.
Suppose by contradiction that there exist $0 \ne a,b \in A$ such that 
  $ab = 0$.
  We may assume that such elements are  homogeneous and have the smallest possible degrees.
Then, 
$ab=0$ in $A/(f)$,  so either $a$ or $b$ is a multiple of $f$, say, $a=fc$ with $0 \ne c$ homogeneous of degree less than $a$. Since $f$ is a regular element, we deduce that $cb=0$, contradicting the choice of $a,b$.
For (2),
see for instance 
\cite[Remark 2.3]{HunekeUlrich}.
\end{proof}

\subsection{Generically reduced}
Recall that $\X(d,n)$ is generically reduced if there exist a dense open subset where the scheme is reduced.
This is equivalent to Serre's property $R_0$ for $\R(d,n)$.

\begin{cor}\label{CorGenericallyReduced}
The scheme $\X(d,n)$ is generically reduced for all $d,n$.
\end{cor}

\begin{proof}
By Lemma \ref{LemmaAlgebraicClosure} (2), in order to show property $R_0$,
we may assume that  $\kk$ is algebraically closed.
Combining Theorems \ref{ThmIrreducibleDecomposition}, \ref{ThmSmoothPointEveryStratum} and Remark \ref{RemTransitiveIrreducibleComponentsSpq},
we see that every irreducible component of $\X(d,n)$ contains a smooth point, therefore, the non-singular locus is dense in $\X(d,n)$.
\end{proof}

\subsection{Complete intersection}
Recall the following definition from the Introduction:
$$
D_{\mathrm{CI}}(n)  := 
\min \left(2 \left \lceil \frac{2n+1 - \sqrt{8n+1}}{2} \right\rceil
\,,\, \,
2 \left \lceil \frac{2n - \sqrt{8n-7}-1}{2} \right\rceil +1\right).
$$
\begin{thm}\label{ThmCIExtended}
Assume $d \geq 2$. The following conditions are equivalent:
\begin{enumerate}
\item $d \geq D_{\mathrm{CI}}(n)$,
\item $\R(d,n)$ is a complete intersection,
\item $\R(d,n)$ is a Gorenstein ring,
\item $\R(d,n)$ is a Cohen-Macaulay ring,
\item $\X(d,n)$ is equidimensional.
\end{enumerate}
Moreover, in this case, $\X(d,n)$ is reduced.
\end{thm}

\begin{proof}
By Lemma \ref{LemmaAlgebraicClosure}, we may assume without loss of generality that $\kk = \overline{\kk}$.

The ideal $\I(d,n)$ is generated by a regular sequence if and only if its height equals its number of minimal generators, 
equivalently,
$\dim \Mat(d,n) - \dim \V(d,n) = \mu \big(\I(d,n)\big)$.
By Proposition \ref{PropMaximum},
this happens if and only if 
$d \geq  2n+1 - \sqrt{8n+1}$ when $d$ is even, 
if and only if
$d \geq  2n - \sqrt{8n-7}$ when $d$ is odd.
By Lemma \ref{LemmaArithmetic} (1), this happens if and only if $d \geq D_{\mathrm{CI}}(n)$.
Thus,
 (1) is equivalent to (2). 

The implications $(2) \Rightarrow (3)  \Rightarrow (4)  \Rightarrow (5) $ are well known.
To show $(5) \Rightarrow (1)$, 
suppose  $d < D_{\mathrm{CI}}(n)$.
By  Lemma \ref{LemmaArithmetic} (1) and Proposition  \ref{PropMaximum},
this implies that $\sigma(p,q)$ attains its maximum in $\Delta(d,n)$ only at $P_2 =  \left(
 d -2\left\lfloor\frac{d}{2} \right\rfloor
 \,,\, 
 \left\lfloor\frac{d}{2} \right\rfloor
 \right).
$
By Remark \ref{RemPrincipalComponent}, the element 
$P_1 = \big(\min(d, n),\,0\big)\in \Delta(d,n)$ is also  maximal with respect to $\preceq$, and we have $P_1 \ne P_2$ since $d \geq 2$.
By Remark \ref{RemIrreducibleComponentsMaximal}, both $P_1$ and $P_2$ give rise to irreducible components of $\X(d,n)$, of different dimension, so $\X(d,n)$ is not equidimensional.

Finally, assume the equivalent conditions hold.
In particular,  $\R(d,n)$ is Cohen-Macaulay, so it satisfies Serre's property $S_1$.
It also satisfies Serre's property $R_0$, by Corollary \ref{CorGenericallyReduced}.
By Serre's criterion for reducedness, it follows that $\R(d,n)$ is reduced.
\end{proof}

\begin{remark}\label{RemI1N}
The case $d=1$ is  exceptional;
the first paragraph of the proof is still valid, but the second is not. 
In any case, the ideal $\I(1,n)$ is not so interesting:
it can be regarded as the Stanley-Reisner ideal of a simplicial complex consisting  of $n$ isolated vertices,
or as the monomial edge ideal of the complete graph on $n$ vertices.
It is easy to see that it is always radical and Cohen-Macaulay,  but it is Gorenstein if and only if $n \leq 2$. 
Formally,
this means that conditions (1), (2) and (3) remain equivalent also when $d=1$, but $(4)$ and $(5)$ are strictly weaker.
\end{remark}

\subsection{Domain}
Recall the following definition from the Introduction:
$$
D_{\mathrm{prime}}(n)  :=\min \left(2 \left \lfloor \frac{2n+1 - \sqrt{8n+1}}{2} \right\rfloor +2
\,,\, \,
2 \left \lfloor \frac{2n - \sqrt{8n-7}+1}{2} \right\rfloor +1\right),
$$
\begin{thm}\label{ThmPrimeExtended}
The following conditions are equivalent:
\begin{enumerate}
\item $d \geq D_{\mathrm{prime}}(n)$,
\item $\R(d,n)$ is a domain,
\item $\R(d,n)$ is a normal domain.
\end{enumerate}
\end{thm}

\begin{proof}
We first prove the theorem when $\kk$ is algebraically closed.

We begin by showing 
that $\V(d,n)$ is irreducible if and only if $d \geq D_{\mathrm{prime}}(n)$.

Consider the first endpoint $P_1 = \big(\min(d,\,n)\,,\,0\big)  \in \Omega$, cf. \eqref{EqEndpoints}.
By Remarks \ref{RemIrreducibleComponentsMaximal} and \ref{RemPrincipalComponent}, 
the closures of the irreducible components of the stratum $\mcS_{P_1}$ are irreducible components of $\V(d,n)$.
Thus, $\V(d,n)$ is irreducible if and only if $\mcS_{P_1}$ is irreducible and $P_1$ is the unique maximal element of $\Delta(d,n)$.
It follows 
from Corollary \ref{CorStructureSpq} that $\mcS_{P_1}$ is irreducible if and only if $d \geq n$. 
Since we assume $n\geq 2$ throughout, if $d \geq n$ then we have $P_1 \ne P_2$, cf. \eqref{EqEndpoints}.
We conclude,
by Proposition \ref{PropMaximum},
that if $\V(d,n)$ is irreducible then 
$d >  2n+1 - \sqrt{8n+1}$ if $d$ is even, 
and
$d >  2n - \sqrt{8n-7}$ if $d$ is odd,
equivalently,
by Lemma \ref{LemmaArithmetic} (2), we have $d \geq D_{\mathrm{prime}}(n)$.

Conversely,
assume that $d \geq D_{\mathrm{prime}}(n)$.
Since $D_{\mathrm{prime}}(n)\geq n $ for all $n \in \ZZ_{\geq 2}$,
we have $d \geq n$ and   $\mcS_{P_1}$ is irreducible.
Since $d \geq D_{\mathrm{prime}}(n)$, by Lemma \ref{LemmaArithmetic} (2) and Proposition \ref{PropMaximum} the function $\sigma(p,q)$ attains the maximum in $\Delta(d,n)$ only at $P_1$, and this maximum is equal to $dn-{n \choose 2}$.
By Theorem \ref{ThmIrreducibleDecomposition}, this implies that $P_1$ is the only maximal element of $\Delta(d,n)$, and this shows that $\V(d,n)$ is irreducible.

This implies $(3) \Rightarrow (2) \Rightarrow (1)$. 
Next, assuming that  (1)  holds, we will show that $\R(d,n)$ is a normal domain, using Serre's criterion for normality.
Since $D_{\mathrm{prime}}(n) \geq D_{\mathrm{CI}}(n)$ (this follows immediately from Lemma \ref{LemmaArithmetic}), (1) implies that $\R(d,n)$ is Cohen-Macaulay by Theorem \ref{ThmCIExtended}, so it satisfies Serre's property $S_2$.
It remains to show that $\R(d,n)$ satisfies Serre's property $R_1$, that is, $\X(d,n)$ is nonsingular in codimension 1.

By irreducibility,
$\mcS_{n,0}=\mcL_{n,0}$ is the unique dense stratum, and it contains a smooth point of    $\X(d,n)$ by Theorem \ref{ThmSmoothPointEveryStratum}.
By Corollary \ref{CorLpqHomogeneous}, it is a homogeneous space, 
so it contains no singular points of $\X(d,n)$.
Any other stratum $\mcS_{p,q}$ with $(p,q) \in \Omega$ has smaller dimension, and is generically smooth in $\X(d,n)$ by 
Theorem \ref{ThmSmoothPointEveryStratum},
so any singular points in these strata form a subset of $\X(d,n)$  of codimension at least 2.

In order to conclude the proof,
it remains to show that any stratum  $\mcS_{p,q}$ with $(p,q) \in \Delta(d,n) \setminus \Omega$ and with 
codimension 1 in $\X(d,n)$ is generically smooth.
Assume such $(p,q)$ exists.
By Proposition \ref{PropNecessaryConditionsPoset} (1),  $(p,q)$ cannot be dominated, with respect to $\preceq$, by any element of $\Delta(d,n)$ other than $(n,0)$,
otherwise the codimension would be larger than 1.
This implies that $(p,q) = (n-1,0)$ and, by Proposition \ref{PropDeformQ}, that $(n-1,1) \notin \Delta(d,n)$.
The latter condition implies that  $(n-1) +2(1) >d$; since
$d \geq n$, it follows that $d = n$.
We have $D_{\mathrm{prime}}(n)\geq n +1$ for all $n \geq 6$,
therefore, we are left with showing that   $\mcL_{n-1,0}$ contains a smooth point of $\X(d,n)$, when $n = 2, \ldots, 5$ and $d=n$.
Since we know that $\X(n,n)$ is irreducible of codimension ${n \choose 2}$ if $2 \leq n \leq 5$, 
we can use Proposition \ref{PropReductionToP0} and, by induction, it suffices to show 
this for $n=2$.
But this is trivial, since $\X(2,2)$ is a cone over a smooth quadric in $\mathbb{P}^3$.

Now, let $\kk$ be an arbitrary field. 
By Lemma \ref{LemmaAlgebraicClosure} (1), 
the direction
 $(1) \Rightarrow (3)$ follows from the algebraically closed case.
Since $(3) \Rightarrow (2)$ is trivial, it remains to show 
that $(2) \Rightarrow (1)$.
Assume that $d < D_{\mathrm{prime}}(n)$, we will show that $\I(d,n)$ is not prime.
If $d = 1$ this is clear, since $n \geq 2$;
thus, assume $d \geq 2$. 
In particular, we have $P_1 \ne P_2$, cf. \eqref{EqEndpoints}.
By  Proposition \ref{PropMaximum} and Lemma \ref{LemmaArithmetic} (2), 
$d < D_{\mathrm{prime}(n)}$ implies that
$\sigma(p,q)$ attains the maximum at $P_2$,
thus, the closures of the irreducible components of $\mcS_{P_2}$ are irreducible components of $\X(d,n) \times_{\kk} \overline{\kk}$.
By Remark \ref{RemPrincipalComponent},
the same holds for $\mcS_{P_1}$.

For any $A\in \mcS_{P_1}$, we have  $\rk(A) = \min(d,n)$.
For any $B\in \mcS_{P_2}$, we have
$
\rk(B) = 
 d -\left\lfloor\frac{d}{2} \right\rfloor = \left\lceil\frac{d}{2} \right\rceil,
$
which is strictly smaller than $d$ since $d \geq 2$.
We also have $\left\lceil\frac{d}{2} \right\rceil < n$, since 
 $D_{\mathrm{prime}}(n) \leq 2n-2$ for all $n \in \mathbb{Z}_{\geq 2}$,
and thus  $d \leq 2n-3$.
In conclusion, $\rk(B) < \rk(A) = \min(d,n) $ for all $A\in \mcS_{P_1}$,
$B \in \mcS_{P_2}$.

Consider the generic $d\times n$ matrix $(x_{i,j})$, 
and let $J\subseteq S$ be the ideal generated by its  minors of size $ \min(d,n)$.
The previous two paragraphs imply that the colon ideal 
 $\big(\I(d,n) \otimes_\kk \overline{\kk}\big) : \big(J \otimes_\kk \overline{\kk}\big)$ is strictly contained between 
  $\I(d,n) \otimes_\kk \overline{\kk}$ and $S\otimes_\kk\overline{\kk}$,
  since it cuts out some but not all of the irreducible 
  components of $\X(d,n) \times_{\kk} \overline{\kk}$.
  However, colon ideals commute with field extensions, 
thus,   
we must have strict inclusions $\I(d,n) \subsetneq \I(d,n) : J \subsetneq S$,
  and therefore $\I(d,n)$ is not a prime ideal.
\end{proof}

\subsection{Unique factorization domain}
Recall the following definition from the Introduction:
$$
D_{\mathrm{UFD}}(n)  :=
\begin{cases}
\min \left(
2 \left\lceil \frac{ 2n+1 - \sqrt{8n-23}}{2}\right\rceil
\,,\, \,
2 \left\lceil \frac{ 2n -1- \sqrt{8n-31}}{2}\right\rceil+1
\right) & \quad \text{if } n \geq 4,\\
4=D_{\mathrm{prime}}(3)+1 & \quad\text{if } n =3,\\
3=D_{\mathrm{prime}}(2)+1 & \quad\text{if } n =2.
\end{cases}
$$
\begin{thm}\label{ThmUFDBody}
If $d \geq D_{\mathrm{UFD}}(n)$, then  $\R(d,n)$ is a unique factorization domain.
\end{thm}
\begin{proof}
Assume $d \geq D_{\mathrm{UFD}}(n)$ throughout the proof.
We observe that $D_{\mathrm{UFD}}(n) \geq 
D_{\mathrm{prime}}(n) \geq
D_{\mathrm{CI}}(n)$ for all $n \in \ZZ_{\geq 2}$
(this follows immediately from Lemma \ref{LemmaArithmetic}).
Thus,   $\R(d,n)$ is a complete intersection.
By the Samuel-Grothendieck theorem, in order to show that $\R(d,n)$ is a UFD, 
it suffices to show that $\R(d,n)$ satisfies Serre's property $R_3$,
that is, $\X(d,n)$ is nonsingular in codimension 3.
By Lemma \ref{LemmaAlgebraicClosure}, we may assume that $\kk = \overline{\kk}$.
Observe that, by Theorem \ref{ThmPrimeExtended},
 the principal component 
$\overline{\mcL_{n,0}}$
is the only irreducible component,
thus   $\dim \R(d,n) = nd - {n \choose 2}$
  and $\Omega_1 \ne \varnothing$.

If $n = 2$, then $d \geq D_{\mathrm{UFD}}(2)=3$,
and the conclusion is obvious since 
$\X(d,2)$ is a cone over a smooth quadric hypersurface in $\mathbb{P}^{2d-1}$,
so, its only singular point is the origin of $\Mat(d,2)$.
The case $n= 3, 
d = D_{\mathrm{UFD}}(3)=4$ is treated separately, at the end of the proof.
Therefore, we assume $n \geq 3$ and $(d,n) \ne (4,3)$ from now on.

In terms of the stratification of Section \ref{SectionStratification}, 
nonsingularity in codimension 3 will follow once we show that, for each  $(p,q) \in \Delta(d,n)$,  one of the following conditions holds:
\begin{itemize}
\item every irreducible component of  $\mcS_{p,q}$ is a homogeneous space and contains a smooth point of $\X(d,n)$;
\item every irreducible component of $\mcS_{p,q}$ has codimension at least 3 in $\V(d,n)$ and 
contains a smooth point of $\X(d,n)$;
\item $\mcS_{p,q}$ has codimension at least 4 in  $\V(d,n)$.
\end{itemize}
In turn, by Definition \ref{DefStratumLpq},  Corollaries \ref{CorLpqHomogeneous} and 
Theorem \ref{ThmSmoothPointEveryStratum}, it suffices to show,
for each $(p,q) \in \Delta(d,n)$,
 one of the following conditions:
\begin{enumerate}
\item $p+q=n$;
\item $q=0$ and $\mcL_{p,0}$ contains a smooth point of $\X(d,n)$;
\item $\sigma(p,q) \leq  nd - {n \choose 2} -3$ and $\mcL_{p,q}$ contains a smooth point of $\X(d,n)$;
\item $\sigma(p,q) \leq nd - {n \choose 2} -4$.
\end{enumerate}

We are going to prove the following 

\noindent
\underline{Claim}:
if $(p,q) \in \Omega_2 \setminus\{(n-2,2), (n-1,1), (n,0)\}$,
the codimension of $\mcL_{p,q}$  in $\V(d,n)$  is  at least 3.

First,
we show that the Claim implies the desired conditions.
If $(p,q)\in\{(n-2,2), (n-1,1), (n,0)\}$, then condition (1) holds. 
If $(p,q)\in  \Omega \setminus\{(n-2,2), (n-1,1), (n,0)\}$, then 
$\mcL_{p,q}$ has codimension   at least 3 in $\V(d,n)$:
this follows by the Claim if $(p,q) \in \Omega_2$, and by Proposition \ref{PropRestrictionSegment1} if $(p,q) \in \Omega_1$.
By Corollary  \ref{CorGenericallyReduced}, this verifies condition (3) for all $(p,q)\in  \Omega \setminus\{(n-2,2), (n-1,1), (n,0)\}$.
Moreover, by Propositions \ref{PropNecessaryConditionsPoset} (1), \ref{PropDeformP}, \ref{PropDeformQ}, 
this also implies condition (4) for all $(p,q)\in \Delta(d,n) \setminus \Omega$ such that $q \geq 3$.
The locus $\mcL_{n-1,0}$ satisfies condition (2)
 by Corollary \ref{CorLpqHomogeneous} and the proof of Theorem \ref{ThmPrimeExtended}.

It remains to consider those $(p,q)\in \Delta(d,n) \setminus \Omega$ such that $q \leq 2$ and $p \leq n-2$.
All such $(p,q)$ are dominated by $(n-2,1)$ or $(n-3,2)$.
We compute the codimensions
\begin{align*}
\codim\, \mcL_{n-2,1} &= \sigma(n,0)-\sigma(n-2,1) =  d-n+1\\
\codim\, \mcL_{n-3,2} &= \sigma(n,0)-\sigma(n-3,2) =  d-n+1
\end{align*}
We have
$D_{\mathrm{UFD}}(n) \geq n+3$ if $n > 7$, 
while $D_{\mathrm{UFD}}(n)=n+2$ for $4 \leq n \leq 7$, and $D_{\mathrm{UFD}}(3)=4$. 
Thus, condition (4) follows for all the remaining pairs $(p,q)$,
except for $(p,q) = (n-2,1),(n-3,2)$ 
when $(d,n) = (5,3), (6,4), (7,5), (8,6), (9,7)$, 
for which  the codimension is 3.
In the  exceptional cases, 
we verify condition (3) instead.
Using Proposition \ref{PropReductionToP0}, 
as in the proof of Theorem \ref{ThmPrimeExtended},
we may reduce the cases $(p,q) = (n-2,1)$
 to $(d,n) = (4,2)$ and $(p,q) = (0,1)$;
the claim follows, since the only singular point of $\X(4,2)$ is $0 \in \Mat(4,2)$.
Similarly, the cases  $(p,q) = (n-3,2)$
reduce to $(d,n) = (5,3)$ and $(p,q) = (0,2)$;
the claim follows from the case $(d,n) = (4,3)$ and $(p,q) = (0,2)$, treated at the end of the proof.

This concludes the proof of the fact that the Claim above implies the statement of the theorem.
Now, we are going to prove the Claim.
The statement is vacuous if 
 $\Omega_2 = \varnothing$,
   so we assume  $\Omega_2 \ne \varnothing$.
As observed earlier, we also have $\Omega_1 \ne \varnothing$, thus, by 
\eqref{EqTable3CasesSegments},
we have $n \leq d \leq  2n-2$.
In fact, as observed above,
we have 
$n+2 \leq D_{\mathrm{UFD}}(n) \leq d \leq 2n-2,$
and this implies, in particular, that $n \geq 4$.
By Proposition \ref{PropRestrictionSegment2}, it suffices to show that 
  $\sigma(P_2) \leq  nd - {n \choose 2}-3$.
  Recall
$P_2 =
 \left(
 d -2\left\lfloor\frac{d}{2} \right\rfloor
 \,,\, 
 \left\lfloor\frac{d}{2} \right\rfloor
 \right).$
 
Assume $d$ is even. 
The assumption 
 $d \geq D_{\mathrm{UFD}}(n)$ is equivalent to $d \geq 2n+1 - \sqrt{8n-23}$, 
 by Lemma \ref{LemmaArithmetic}. 
 Rewriting this as $ 2n+1-d \leq \sqrt{8n-23} $ and squaring,
 we obtain an inequality that is equivalent to 
$$
\sigma(P_2) =\sigma \left(0\,, \frac{d}{2}\right)= \frac{1}{8}d^2+\frac{1}{2}dn-\frac{1}{4}d  \leq nd - {n \choose 2}-3,
$$
as desired.
Similarly, if $d$ is odd, then
 $d \geq D_{\mathrm{UFD}}(n)$ is equivalent to $d \geq 2n - \sqrt{8n-31}$, 
 by Lemma \ref{LemmaArithmetic}. 
 Rewriting this as $ 2n-d \leq \sqrt{8n-31} $ and squaring,
 we obtain an inequality that is equivalent to 
 $$
\sigma(P_2) =\sigma \left(1\,, \frac{d-1}{2}\right)= \frac{1}{8}d^2+\frac{1}{2}dn-\frac{1}{2}n+\frac{7}{8}
\leq nd - {n \choose 2}-3,
$$
and this concludes the proof of the theorem, except for the case $(d,n) = (4,3)$.
Note that, 
in applying Lemma \ref{LemmaArithmetic}, we used the fact that 
$$
\Big|\big(2n+1 - \sqrt{8n-23}\big)-\big(2n - \sqrt{8n-31}\big) \Big| \leq 1 $$
 for all $n \geq 4$, which can be easily verified.
 
 Finally, we turn to the proof of the $R_3$ property for the case $(d,n) = (4,3)$.
  We only need to inspect the strata corresponding to  $(p,q) = (2,0),  (1,1), (0,2), (1,0), (0,1)$.
One checks immediately that 
$(1,0)$ and $(0,1)$ satisfy condition (4).
We are going to show that $\mcL_{2,0}, \mcL_{1,1}$ and $\mcL_{0,2}$ contain no singular points of $\X(4,3)$. 
Fix a square root $\nu \in \kk$ of $-1$.
By acting with the group $\Ort(4) \times (\T(3)\rtimes \Sigma_{3-p})$, any point of $\mcL_{2,0}, \mcL_{1,1}, \mcL_{0,2}$ can be transformed, respectively, into 
$$
A_{2,0}=\matrix{1 & 0 & 0 \\ 0 & 1 & 0 \\ 0 & 0 & 0 \\ 0 & 0 & 0},
\quad
A_{1,1 }(\delta)=\matrix{1 & 0 & 0 \\ 0 & 1 & \delta \\ 0 & \nu & \nu \delta \\ 0 & 0 & 0},
\quad
A_{0,2}(\varepsilon,\eta)=\matrix{1 & 0 & \varepsilon \\ \nu & 0 & \nu\varepsilon \\ 0 & 1 & \eta \\ 0 & \nu & \nu \eta},
$$
for some $\delta, \varepsilon,  \eta \in\kk$.
Using the Jacobian matrix
$$
\Theta(4,3)
=
\begin{pNiceMatrix}[first-col,first-row]
&\shortstack{ $x_{1,1}$\\ \,} &\shortstack{ $x_{1,2}$\\ \,} &\shortstack{ $x_{1,3}$\\ \,} &\shortstack{ $x_{2,1}$\\ \,} &\shortstack{ $x_{2,2}$\\ \,} &\shortstack{ $x_{2,3}$\\ \,} &\shortstack{ $x_{3,1}$\\ \,} &\shortstack{ $x_{3,2}$\\ \,} &\shortstack{ $x_{3,3}$\\ \,} &\shortstack{ $x_{4,1}$\\ \,} &\shortstack{ $x_{4,2}$\\ \,} &\shortstack{ $x_{4,3}$\\ \,}
\\
(1,2)\,\, & x_{1,2} & x_{1,1} & 0  
& x_{2,2} & x_{2,1} & 0  
& x_{3,2} & x_{3,1} & 0  
& x_{4,2} & x_{4,1} & 0  
\\
(1,3) \,\,& x_{1,3} & 0 &  x_{1,1} &
x_{2,3} & 0 &  x_{2,1} &
x_{3,3} & 0 &  x_{3,1} &
x_{4,3} & 0 &  x_{4,1} 
\\
(2,3)\,\, & 0 & x_{1,3} & x_{1,2} &
0 & x_{2,3} & x_{2,2} & 
0 & x_{3,3} & x_{3,2} & 
0 & x_{4,3} & x_{4,2}  
\end{pNiceMatrix}
$$
it follows that these points are smooth on $\X(4,3)$, for all $\delta,\varepsilon, \eta \in \kk$.
\end{proof}

 \subsection{Applications to Lov\'asz-Saks-Schrijver ideals}
Recall the definition of the $d$-th LSS ideal of a simple graph $\mcG$  on $n$ vertices.
$$
\LSS(d,\mathcal{G}) = 
\left( \sum_{i=1}^d x_{i,j}x_{i,k} \, : \, (j,k) \text{ is an edge of } \mathcal{G}\right)\subseteq S.
$$ 

\begin{thm}
Let $\mcG$ be a simple graph on $n$ vertices.
\begin{enumerate}
\item  If $d \geq D_{\mathrm{CI}}(n)$, then 
$\LSS(d,\mathcal{G})$ is a radical complete intersection;
\item if $d \geq D_{\mathrm{prime}}(n)$, then   $S/\LSS(d,\mathcal{G})$ is a normal domain;
\item if $d \geq D_{\mathrm{UFD}}(n)$, then $S/\LSS(d,\mathcal{G})$ is a UFD.
\end{enumerate}
\end{thm}

\begin{proof}
Since 
$D_{\mathrm{UFD}}(n) \geq 
D_{\mathrm{prime}}(n) \geq
D_{\mathrm{CI}}(n)$ for all $n \in \ZZ_{\geq 2}$,
we may assume that $\I(d,n)$ is a complete intersection.
Then, clearly $\LSS(d, \mathcal{G})$ is also a complete intersection, since it is generated by a subset of the generators of $\I(d,n)$.
Moreover, $S/\LSS(d,\mathcal{G})$ is a deformation of $\R(d,n)$,
in the sense that the latter is obtained from the former by going modulo a regular sequence.
If $d \geq D_{\mathrm{prime}}(n)$, then   
$\R(d,n)$ is a domain, so, 
by Lemma \ref{LemmaDeformation} (1),
$S/\LSS(d,\mathcal{G})$ is also a domain.
By Lemma \ref{LemmaDeformation} (2),
in statements (1), (2), (3),
$S/\LSS(d,\mathcal{G})$ inherits from $\R(d,n)$ conditions $R_0, R_1, R_3$, respectively,
and the conclusions follow from Serre's criteria for reducedness, for normality, and from the Samuel-Grothendieck theorem, respectively,
as in Theorems \ref{ThmCIExtended}, \ref{ThmPrimeExtended}, \ref{ThmUFDBody}.
\end{proof}

\section{Open problems}\label{SectionFutureDirections}

 In this final section,
we discuss  potential future directions, and collect some open problems suggested
by our work.
Assume, for simplicity, that $\kk=\overline{\kk}$.

\subsection{Rational singularities}
By Theorem \ref{ThmPrimeExtended},
the variety $\V(d,n)$ is normal and Cohen-Macaulay as soon as it is irreducible.
A stronger condition, which implies both normal and Cohen-Macaulay, is  having  rational singularities ($F$-rational if $\mathrm{char}(\kk) >0$).
This condition is relevant for applications to moduli spaces and representation theory \cite{AA,K}. 
In \cite[Theorem A]{TV}, a sufficient condition for LSS rings to have ($F$-)rational singularities is given, in terms of the positive matching decomposition and of the degeneracy number of a graph;
it implies that $\V(d,n)$ has ($F$-)rational singularities if $d \geq 3n-4$.
It is natural to ask if the following sharper bound holds.

\begin{que}
Does $\V(d,n)$ have ($F$-)rational singularities if $d \geq D_{\mathrm{prime}}(n)$?
\end{que}

\subsection{Syzygies and representation theory}
Theorem \ref{ThmCI} says that $\I(d,n)$ is resolved by the Koszul complex if and only if $d \geq D_{\mathrm{CI}}(n)$.
In general, the free resolution or the graded Betti numbers of 
$\I(d,n)$ are unknown.

\begin{que}
What is the Betti table of $\I(d,n)$?
\end{que}

The  case $d=1$ is rather trivial, cf. Remark \ref{RemI1N}.
The case $d=2$  was solved   recently in \cite{GHSW, JZ} when $\mathrm{char}(\kk) =0$;
recall that the ideal $\I(2,n)$ is the same, up to a change of coordinates,
as the ideal of $2\times 2$ permanents of a generic $2 \times n$ matrix,
cf. \cite[Section 3]{HMMW}.

In fact, since $\I(d,n)$ is invariant under the group $G$ of \eqref{EqGroupG},
it would be interesting to determine the equivariant free resolution of $\I(d,n)$ with respect to $G$ (or some of its subgroups).
In other words, 
to determine the decomposition of the 
syzygy modules 
$\mathrm{Tor}^S_i(\I(d,n),\kk)$
into irreducible representations.
This was done for $\I(2,n)$ and the group $\Sigma_2 \times \Sigma_n$ in \cite[Section 4.2]{GHSW}, for the first syzygy module, and then in \cite{JZ} for all syzygies.

\subsection{Distinguished components of the variety of orthogonal frames}\label{SubsectionComponents}
It is natural to  investigate 
certain (unions of) irreducible components of $\V(d,n)$,
and, in particular, to study for them  some of the  problems discussed so far for $\V(d,n)$.

To begin,
consider the purely isotropic locus.
Suppose, for simplicity, that $d$ is even and $d \leq 2n$.
Then,
$$\V_{\iso}(d,n)= \overline{\mcS_{0,q}}\subseteq \V(d,n)
\qquad
\text{with}
\qquad
q = \frac{d}{2}
$$
is the variety of  isotropic  $n$-frames of $\kk^d$, that is,
$\V_{\iso}(d,n)=\big\{ A \in \Mat(d,n) \, : \, A^\intercal A=0\big\}$.
By Corollary \ref{CorStructureSpq}, 
it is
the union of two isomorphic irreducible components. 
It follows by the proof of Theorem \ref{ThmPrimeExtended} that if $d < D_{\mathrm{prime}}(n)$ then the components of  $\V_{\iso}(d,n)$ are also  components of $\V(d,n)$.

This variety is studied in \cite[p. 219]{Weyman},
by means of the Kempf-Lascoux-Weyman geometric technique.
In fact, $\V_{\iso}(d,n)$
 can be regarded as an analogue of determinantal varieties for the orthogonal group $\Ort(d)$.
It is shown that the defining ideal of $\V_\iso(d,n)$ 
is generated by the  $(q+1)$-minors of the generic matrix $(x_{i,j})$ and by the  ${n+1 \choose 2}$ quadrics corresponding to the matrix equation $A^\intercal A = 0$,
and that its two irreducible components are normal and Cohen-Macaulay varieties with rational singularities.
See also \cite[Section 7]{HJPS}.

Now,
we discuss the purely anisotropic locus.
Suppose, for simplicity, that $d \geq n$.
The set 
$$
\V_{\ani}(d,n) = 
\overline{\mcS_{n,0}}
\subseteq \V(d,n)
$$
is the principal irreducible component of $\V(d,n)$, 
which generically parametrizes linearly independent orthogonal frames,
cf. Definition \ref{DefPrincipalComponent}.
In particular, when $d=n$,  $\V_{\ani}(d,d)$ is the closure of the set of orthogonal bases of $\kk^d$, 
as opposed to the orthogonal group $\Ort(d)$, which is the set of orthonormal bases.

In contrast with $\V_\iso(d,n)$,
very little is known about the variety $\V_\ani(d,n)$.
If $d \geq D_\mathrm{prime}(n)$,
 then $\V_\ani(d,n) = \V(d,n)$ by Theorem \ref{ThmPrime}, so it is a normal variety, 
and if $d \geq 3n-4$ 
then it has rational singularities by \cite[Theorem A]{TV}.
But nothing is known about $\V_\ani(d,n)$ when it differs from the variety of orthogonal frames.

\begin{que}
What is the defining ideal of $\V_{\ani}(d,n)$?
\end{que}

This question is  open already in the first case when 
$\V_\ani(d,n) \ne \V(d,n)$, which is $(d,n) = (6,6)$.

It is natural to ask whether the singularities of $\V_\ani(d,n)$ are good also when $d < D_{\mathrm{prime}}(n)$.

\begin{que}
Is the variety $\V_{\ani}(d,n)$ normal?
\end{que}

\subsection{Residual intersection and linkage}
In addressing the questions of the previous subsection, 
it may be interesting to study the variety of orthogonal frames from the points of view of the theory of linkage and residual intersections.
For the purpose of this discussion, assume that Conjecture \ref{ConjectureRadical} has a positive answer, that is, $\I(d,n)$ is a radical ideal.
If $J$ denotes the defining ideal of  $\V(d,n)\setminus \V_\ani(d,n)$,
 then $\V_\ani(d,n)$ is defined  by the colon ideal $\I(d,n):J$.
Since $\V_\ani(d,n)$ has codimension equal to the number of generators of $\I(d,n)$, 
 cf. \eqref{EqCodimPrincipal},
 this means that $\I(d,n):J$ is a residual intersection of $J$, in the sense of
\cite{ArtinNagata,HunekeUlrich2}.

When
$d = D_{\mathrm{CI}}(n)<D_{\mathrm{prime}}(n)$,
then 
 $\V(d,n)$ is a complete intersection but it is reducible.
 Observe that, in this case, we do know that $\I(d,n)$ is radical.
 Assume for simplicity that
 $d$ is even, then, by the proof of Theorem \ref{ThmCIExtended}, we have $\V(d,n) = \V_\ani(d,n) \cup \V_\iso(d,n)$.
In this case, 
residual intersection simply corresponds to linkage
\cite{HunekeUlrich,PeskineSzpiro},
and the defining ideal of $\V_\ani(d,n)$ is the link of the defining ideal of $\V_\iso(d,n)$ via the regular sequence $\I(d,n)$.
This happens, for example, for $(d,n) = (6,6)$.
In general, $\V(d,n)$ is not equidimensional, and $\V_\ani(d,n)$ is only a residual intersection.
For example, when $(d,n) = (10,9)$, 
then 
$\V(10,9) = \V_\ani(10,9) \cup \V_\iso(10,9)$
 by Theorem \ref{ThmIrreducibleDecomposition},
with $\dim \V_\ani(10,9) = 54, \dim  \V_\iso(10,9)= 55$,
thus, $\V_\ani(10,9)$ is a residual intersection of $\V_\iso(10,9)$,
via the ideal $\I(10,9)$.
In more complicated examples, there will be many more irreducible components, with intermediate dimensions,
as dictated by Theorem \ref{ThmIrreducibleDecomposition}.

\subsection{Koszul algebras}

The following question was suggested by Jason McCullough.

\begin{que}
When is $\R(d,n)$ a Koszul algebra?
\end{que}

Recall that a Noetherian standard graded $\kk$-algebra $A$ is called \emph{Koszul} if the residue field $\kk$ has a linear free resolution as $A$-module, equivalently, if $\big[ \mathrm{Tor}^A_i(\kk,\kk)\big]_j  = 0$ whenever $i \ne j$.
This is a very natural condition for quadratic algebras.
Indeed, all Koszul algebras are defined by quadratic relations, and, conversely, several  algebras defined by well-structured quadratic relations are Koszul;
we refer to \cite{CDR} for more information on this subject. 

When $d= 1$, the algebra $\R(1,n)$ is defined by  quadratic monomials, 
and it follows by \cite{Froberg} that it is Koszul.
When $d \geq D_{\mathrm{CI}}(n)$, 
the algebra $\R(d,n)$ is defined by a regular sequence of quadrics, 
and it follows by \cite{Tate} that it is Koszul.
However, 
in general, $\R(d,n)$ may fail to be Koszul.
For example,  $\R(2,4)$  is not Koszul when $\mathrm{char}(\kk) = 0$: 
a calculation  with \cite{M2}
shows that there are 22 syzygies among  the $6$ quadratic generators of $\I(2,4)$,
and they are all quadratic;
by 
\cite[Theorem 3.1]{ACI}, 
this implies that 
 $\R(2,4)$ cannot be Koszul.

\subsection{Embedded primes of LSS ideals}
In order to complete the proof of Conjecture \ref{ConjectureRadical} that $\I(d,n)$ is a radical ideal, 
in view of Theorem \ref{ThmGenericallyReduced},
it would suffice to show that all associated primes of $\I(d,n)$ are minimal primes.
In fact, 
more generally, no example of LSS ideal with embedded components is known.

\begin{que}
Let $\mathcal{G}$ be a simple graph. Are all associated primes of $\LSS(d,\mathcal{G})$ minimal?
\end{que}

\printbibliography

\end{document}